\documentclass[12pt,a4paper,twoside,reqno]{amsart}

\usepackage{amsmath, amssymb, amsthm, enumerate, enumitem}
\usepackage[normalem]{ulem}
\usepackage{xr}
\makeatletter
\def\l@section{\@tocline{1}{10pt}{1pc}{}{}}
\def\l@subsection{\@tocline{2}{0pt}{1pc}{4.6em}{}}
\def\l@subsubsection{\@tocline{3}{0pt}{1pc}{7.6em}{}}
\renewcommand{\tocsection}[3]{%
  \indentlabel{\@ifnotempty{#2}{\makebox[2.3em][l]{%
    \ignorespaces#1 #2.\hfill}}}\textbf{#3}}
\renewcommand{\tocsubsection}[3]{%
  \indentlabel{\@ifnotempty{#2}{\hspace*{2.3em}\makebox[2.3em][l]{%
    \ignorespaces#1 #2.\hfill}}}#3}
\renewcommand{\tocsubsubsection}[3]{%
  \indentlabel{\@ifnotempty{#2}{\hspace*{4.6em}\makebox[3em][l]{%
    \ignorespaces#1 #2.\hfill}}}#3}
\makeatother

\newcommand{\MM}{\mathcal{M}}

\newcommand{\IR}{\mathbb{R}}

\newcommand{\eps}{\varepsilon}
\newcommand{\ov}[1]{\overline{#1}}

\newcommand{\td}[1]{\widetilde{#1}}

\DeclareMathOperator{\scal}{scal}
\DeclareMathOperator{\Int}{Int}

\DeclareMathOperator{\Ric}{Ric}

\DeclareMathOperator{\area}{area}

\DeclareMathOperator{\tr}{tr}

\DeclareMathOperator{\dist}{dist}

\DeclareMathOperator{\vol}{vol}

\DeclareMathOperator{\DIV}{div}

\newcommand{\nangle}{\sphericalangle}

\newcommand{\EMPTY}[1]{}

\newtheorem{Theorem}{Theorem}[section]
\newtheorem{Lemma}[Theorem]{Lemma}

\newtheorem{Proposition}[Theorem]{Proposition}
\newtheorem{Definition}[Theorem]{Definition}

\theoremstyle{definition}
\newtheorem{Remark}[Theorem]{Remark}
\numberwithin{equation}{section}

\title[Long-time behavior of 3d Ricci flow --- B]{Long-time behavior of 3 dimensional Ricci flow\\B: Evolution of the minimal area of simplicial complexes under Ricci flow}
\author{Richard H Bamler}
\address{UC Berkeley, Department of Mathematics, 970 Evans Hall, Berkeley, CA 94720, USA}
\email{rbamler@math.berkeley.edu}
\date{\today}

\begin{document}
\begin{abstract}
In this second part of a series of papers on the long-time behavior of Ricci flows with surgery, we establish a bound on the evolution of the infimal area of simplicial complexes inside a 3-manifold under the Ricci flow.
This estimate generalizes an area estimate of Hamilton, which we will recall in the first part of the paper.

We remark that in this paper we will mostly be dealing with non-singular Ricci flows.
The existence of surgeries will not play an important role.
\end{abstract}

\maketitle
\tableofcontents

\section{Introduction and statement of the results}
Consider a closed $3$-manifold $M$ with $\pi_2(M) = 0$, a finite $2$-dimensional simplicial complex $V$ (see Definition \ref{Def:simplcomplex} below for details), possibly with boundary, and a continuous map $f_0 : V \to M$ such that $f_0$ restricted to each edge of $\partial V$ is a smooth immersion.
Suppose that $(g_t)_{t \in [T_1, T_2]}$, $T_1 > 0$ is a Ricci flow (i.e. $\partial_t g_t = - 2 \Ric_{g_t}$) on $M$ such that $\scal_t \geq - \frac{3}{2t}$ for all $t \in [T_1, T_2]$.
For every $t \in [T_1, T_2]$ let
\[ A_t (f_0) : = \inf \big\{ \area_t f' \;\; : \;\; f' \simeq f_0 \;\; \text{relative to $\partial V$} \big\} \]
be the infimum over the time-$t$ areas of all maps $f' : V \to M$ that are homotopic to $f_0$ relative to $\partial V$.
By this we mean that there is a continuous maps $H : V \times [0,1] \to M$ such that $H(\cdot, 0) = f_0$, $H(\cdot, 1) = f'$ and $H(\cdot , s) = f_0$ on $\partial V$ for all $s \in [0,1]$.
Then the main result of this paper is that in the forward barrier sense\footnote{If $h : [a,b) \to \IR$ is a function, $t_0 \in [a,b)$ and $c \in \IR$, then we say that $\frac{d}{dt^+} |_{t_0} h(t) \leq c$ in the \emph{forward barrier sense} if for any $\delta > 0$ the inequality $h(t) \leq h(t_0) + (c + \delta) (t- t_0)$ holds on an interval of the form $[t_0, t_0 + \tau_{t_0, \delta})$.}
\begin{equation} \label{eq:approximateresultevolutionofA}
\frac{d}{dt^+} A_t (f_0) \leq \frac{3}{4t} A_t (f_0) +  C_t, 
\end{equation}
where $C_t$ is a time-dependent constant that only depends on the topology of $V$ and the geometry of $f_0|_{\partial V}$ with respect to the metric $g_t$ in a controlled way.
We refer to Proposition \ref{Prop:areaevolutioninMM} for more details.
In \cite{Bamler-LT-main}, the result of this paper will be applied to simplicial complexes $V$ and maps $f_0$ as constructed in \cite{Bamler-LT-topology} to prove a conjecture of Perelman.

Consider for a moment the case in which $V$ is a compact surface, possibly with boundary.
In this case the estimate (\ref{eq:approximateresultevolutionofA}) is known, or at least folklore.
It follows from an argument due to Hamilton (cf \cite[sec 11]{Ham}), which makes use of the fact that for every time $t \in [T_1, T_2]$ we can choose a time-$t$ minimal map $f_t : V \to M$ whose area is equal to $A_t(f_0)$.
The argument also makes use of the {Gau{\ss}}-Codazzi equations and the Theorem of Gau\ss-Bonnet.
So for example, in the case in which $V$ is closed, the constant $C_t$ becomes $-2\pi\chi(V)$, where $\chi(V)$ denotes the Euler characteristic of $V$.
We remark that even in the surface case Hamilton's argument is not quite sufficient for our particular setting, since we cannot exclude the existence of branch points, i.e. we cannot guarantee that the minimal map $f_t$ is an immersion.
This issue can however be overcome as we demonstrate in Proposition \ref{Prop:evolminsurfgeneral} below, where we will establish the case in which $V$ is a disk.

Consider the general case in which $V$ is a simplicial complex.
An inspection of the arguments described in the previous paragraph shows that if the existence of an area minimizing map $f_t : V \to M$ is guaranteed, then all of Hamilton's estimates can be carried out.
Here we have to make use of the Euler-Lagrange equations for $f_t$ along the edges of $V$, which state that around every edge the faces meet in directions that add up to zero.
This additional set of equations implies that certain boundary integrals arising in the application of Gau\ss-Bonnet cancel each other out.

Unfortunately, an existence and regularity theory for such minimizers $f_t$ does not exist to the author's knowledge and seems to be difficult to achieve.
We note that, however, if we allow the combinatorial structure of $V$ to vary, then a result of Choe (cf \cite{Choe})---which relies heavily on this fact---states: for every Riemannian metric $g$ on $M$, there is a finite, 2-dimensional simplicial complex $V_g$ and a smooth, minimal embedding $f_g : V_g \to M$ such that the complement of $f_g (V_g)$ is a topological ball.
Such embeddings maps would be interesting in the final part \cite{Bamler-LT-main} or our series, but it seems to be difficult to control the number of vertices of $V_g$ and this number influences the bound $C_t$ in the area evolution estimate of $A_t (f_0)$.
In fact, it is very likely that there are metrics $g_1, g_2, \ldots$ on $M$ for which the number of vertices of the corresponding minimal simplicial complex $V_{g_k}$ diverges.

In order to get around this issue, we will employ the following trick.
Instead of looking for a minimizer of the area functional, we will find a minimizer of the perturbed functional
\begin{equation} \label{eq:perturbedareafunctionalinintro}
 f \longmapsto \area f + \lambda \ell (f|_{V^{(1)}}).
\end{equation}
Here $\lambda > 0$ is a small constant, $\ell (f |_{V^{(1)}})$ denotes the sum of the lengths of $f$ restricted to all edges of $V$ and $f : V \to M$ is any map that is homotopic to $f_0$.
The existence and regularity of a minimizer for the perturbed functional follows now easily (apart from some issues arising from possible self-intersections of the $1$-skeleton).
However, the extra term $\lambda \ell (f |_{V^{(1)}})$ introduces an extra term in the Euler-Lagrange equations along each edge of $V$ and hence the boundary integrals in the evolution estimate for the minimum of this perturbed functional will not cancel each other out, but add up to a new term.
Luckily, it will turn out that this term has the right sign to derive an evolution estimate similar to (\ref{eq:approximateresultevolutionofA}).
Now letting $\lambda$ go to $0$, we obtain the desired evolution estimate for $A_t (f_0)$.

This paper is organized as follows.
In section \ref{sec:spheresanddisks} we present Hamilton's area estimate for spheres (see Proposition \ref{Prop:evolsphere}) and for disks (see Prop \ref{Prop:evolminsurfgeneral}).
Both of these estimates will be needed in \cite{Bamler-LT-main}.
For spheres, Hamilton's argument is straight forward and the computations in this case exhibit the idea underlying the subsequent area estimates very clearly.
For disks, an issue arises due to possible branch points, which can be resolved by a trick.
In section \ref{sec:Defpolygonalcomplex} we define simplicial complexes and section \ref{sec:existenceofvminimizers} contains the existence and regularity discussion for maps from simplicial complexes that minimize the perturbed area functional (\ref{eq:perturbedareafunctionalinintro}).
The results of this section will then be used in section \ref{sec:areaevolutionunderRF} to derive the infimal area evolution estimate for simplicial complexes, i.e. the bound (\ref{eq:approximateresultevolutionofA}).
Proposition \ref{Prop:areaevolutioninMM} in that section will be our main result.
We note that sections \ref{sec:Defpolygonalcomplex}-\ref{sec:areaevolutionunderRF} are independent of section \ref{sec:spheresanddisks}.

Observe that most results in this paper will be phrased in terms of Ricci flows with surgery and precise cutoff $\MM$ as introduced in \cite{Bamler-LT-Perelman}.
The reason for this is that we want to apply these results without change in \cite{Bamler-LT-main}.
However, the possible existence of surgeries is inessential and does not create any issues.
For the purposes of this paper it is only important to be familiar with properties (1) and (5) of the Definition of Ricci flows with surgery and precise cutoff (see \cite[Definition \ref{Def:precisecutoff}]{Bamler-LT-Perelman}).
Property (1) ensures that we have the bound $\sec_t \geq - \frac{3}{2t}$ for all times $t$.
And Property (5) implies that areas cannot increase by a trivial surgery.
Specifically this property implies that for every surgery time $T$ whose surgeries are all trivial and every $\chi > 0$ there is some $t_\chi < T$ such that for all $t \in (t_\chi, T)$ there is a $(1+\chi)$-Lipschitz map $\xi : \MM(t) \to \MM(T)$ that is equal to the identity on the part of the manifold that is not affected by the surgery process at time $T$.
We will be able to conclude from this property that quantities of the type $A_t(f_0)$ are lower semi-continuous in time.

We refer to the introduction of \cite{Bamler-LT-Introduction} for acknowledgements and historical remarks.

\section{Area evolution of spheres and disks} \label{sec:spheresanddisks}
In this subsection we recall area estimates for minimal spheres and disks under Ricci flow.
They were first developed by Hamilton (\cite[sec 11]{Ham}).
The estimates needed in this series of papers are however slightly different from those of Hamilton, which is why we have decided to carry out their proofs.

The first proposition gives us an area estimate for $2$-spheres and will be used in the proof of \cite[Proposition \ref{Prop:irreducibleafterfinitetime}]{Bamler-LT-main} to show that after some time, all time-slices in a Ricci flow with surgery are irreducible and all subsequent surgeries are trivial.
\begin{Proposition} \label{Prop:evolsphere}
Let $\MM$ be a (3 dimensional) Ricci flow with surgery and precise cutoff and closed time-slices, defined on the time-interval $[T_1, T_2]$ $(0 < T_1 < T_2)$.
Assume that the surgeries are all trivial and that $\pi_2 (\MM(t)) \not= 0$ for all $t \in [T_1, T_2]$.
For every time $t \in [T_1, T_2]$ denote by $A(t)$ the infimum of the areas of all homotopically non-trivial immersed $2$-spheres.
Then the quantity
\[ t^{1/4} \big( t^{-1} A(t) + 16 \pi \big) \]
is monotonically non-increasing on $[T_1, T_2]$.
Moreover,
\[ T_2 < \Big( 1 +  \frac{1 }{16 \pi}  T_1^{-1} A(T_1) \Big)^4 T_1. \]
\end{Proposition}
\begin{proof}
Compare also with \cite[Lemmas 18.10 and 18.11]{MTRicciflow}.
Let $t_0 \in [T_1, T_2)$.
By \cite{SU81} and \cite{Gul} or \cite{MY}, there is a non-contractible, conformal, minimal immersion $f : S^2 \to \MM(t_0)$ with $\area_{t_0} f = \area_{S^2} f^* (g(t_0)) = A(t_0)$.
We remark, that using the methods in the proof of Proposition \ref{Prop:evolminsurfgeneral} below, it is enough to assume that $f$ is only smooth.
Call $\Sigma = f(S^2) \subset \MM(t)$.
Then $\Sigma$ is either a $2$-sphere or an $\IR P^2$ with a finite number of self-intersections.
We can estimate the infinitesimal change of the area of $\Sigma$ (we count the area twice if $\Sigma$ is an $\IR P^2$) while we vary the metric in positive time direction (and keep $f$ constant!).
Using the $t_0^{-1}$-positivity of the curvature on $\MM(t_0)$, the fact that the interior sectional curvatures are not larger than the ambient ones as well as Gau\ss-Bonnet, we conclude:
\begin{multline*}
\frac{d}{dt^+} \Big|_{t = t_0} \area_{t} (\Sigma ) = - \int_{\Sigma} \tr_{t_0} (\Ric_{t_0} |_{T \Sigma}) d {\vol}_{t_0} \\
= - \frac12 \int_{\Sigma} \scal_{t_0} d {\vol}_{t_0} - \int_{\Sigma} \sec_{t_0}^{\MM(t_0)}(T \Sigma) d {\vol}_{t_0} 
 \leq \frac3{4t_0} \area_{t_0} (\Sigma) - \int_{\Sigma} \sec^\Sigma d {\vol}_{t_0} \\
 \leq \frac3{4t_0} \area_{t_0}(\Sigma) - 2 \pi \chi(\Sigma) = \frac3{4t_0} A(t_0) - 4 \pi.
\end{multline*}
Here, $\sec^{\MM(t_0)}_{t_0}(T\Sigma)$ denotes the ambient sectional curvature of ${\MM(t_0)}$ tangential to $\Sigma$ and $\sec^\Sigma_{t_0}$ denotes the interior sectional curvature of $\Sigma$.
We conclude from this calculation that $\frac{d}{dt^+}|_{t = t_0} (t^{1/4} (t^{-1} A(t) + 16 \pi )) \leq 0$ in the barrier sense and hence, the quantity $t^{1/4} (t^{-1} A(t) + 16 \pi )$ is monotonically non-increasing in $t$ away from the singular times.

We will now show that $A(t)$ is lower semi-continuous.
We can restrict ourselves to the case in which $t_0$ is a surgery time.
Let $t_k \nearrow t_0$ be a sequence converging to $t_0$ and choose minimal $2$-spheres $\Sigma_k \subset \MM(t_k)$ with $\area_{t_k} \Sigma_k = A(t_k)$.
By property (5) of \cite[Definition \ref{Def:precisecutoff}]{Bamler-LT-Perelman}, we find diffeomorphisms $\xi_k : \MM(t_k) \to \MM(t_0)$ that are $(1+\chi_k)$-Lipschitz for $\chi_k \to 0$.
So $A(t_0) \leq \lim \inf_{k \to \infty} (1+\chi_k)^2 A(t_k) = \lim \inf_{k \to \infty} A(t_k)$.

The lower semi-continuity implies that $t^{1/4} (t^{-1} A(t) + 16 \pi )$ is monotonically non-increasing on $[T_1, T_2]$.
The bound on $T_2$ follows from the fact that $A(T_2) > 0$.
\end{proof}

In the next proposition we estimate the area evolution of minimal disks that are bounded by a given loop of controlled geodesic curvature.
This fact will be used in the main part of the proof of \cite[Theorem \ref{Thm:LT0-main-1}]{Bamler-LT-Introduction}, which can be found in \cite{Bamler-LT-main}, to exclude the long-time existence of short contractible loops as asserted in \cite[Proposition \ref{Prop:structontimeinterval}]{Bamler-LT-main}.
Note that in contrast to Hamilton's setting, we cannot exclude the existence of branch points at the boundary of these disks.
This difference creates some analytical difficulties.

\begin{Proposition} \label{Prop:evolminsurfgeneral}
Let $\MM$ be a (3 dimensional) Ricci flow with surgery and precise cutoff and closed time-slices, defined on the time-interval $[T_1, T_2]$ $(T_2 > T_1 > 0)$.
Assume that all surgeries of $\MM$ are trivial.

Let $\gamma_t  \subset \MM(t)$ be a time-dependent embedded, disjoint loop in $\MM(t)$ that does not hit surgery times and is stationary in time (i.e. between two surgery times, $\gamma_t$ to a fixed loop).
Assume moreover that $\gamma_t$ is contractible in $\MM(t)$ for all $t \in [T_1, T_2]$ and denote by $A(t)$ the infimum of the areas of all smooth maps $f : D^2 \to \MM(t)$ whose restriction to $\partial D^2 = S^1$ parameterizes the loop $\gamma_t$.

Assume that there are constants $\Gamma, a > 0$ such that for all $t \in [T_1, T_2]$
\begin{enumerate}[label=(\roman*)]
\item the geodesic curvatures along $\gamma_t$ satisfy the bound $|\kappa_{\gamma_t, t}| < \Gamma t^{-1/2}$,
\item the length of $\gamma_t$ satisfies the bound $\ell(\gamma_t) < a t^{-1/2}$,
\end{enumerate}

Then the quantity
\[ t^{1/4} \big( t^{-1} A(t) + 4 (2\pi  - a \Gamma ) \big) \]
is non-increasing on $[T_1, T_2]$.

In particular, if $a \Gamma < 2 \pi$, then
\[ T_2 < \Big( 1 +  \frac{ T_1^{-1} A(T_1) }{4 (2\pi - a \Gamma )} \Big)^4 T_1. \]
\end{Proposition}

\begin{proof}
Let $t_0 \in [T_1, T_2]$.
By \cite{Mor} we find a time-$t_0$ area minimizing continuous map $f : D^2 \to \MM(t_0)$ that is smooth on $\Int D^2$ and whose restriction to the boundary $\partial D^2$ parameterizes $\gamma_{t_0}$. Moreover, $f$ is almost conformal and harmonic on $\Int D^2$ and we have $A(t_0) = \area f^*(g(t_0))$.
Next, we use use \cite{HH} to conclude that $f$ is even smooth up to the boundary and an immersion away from finitely many branch points.

Analogously as in the proof of Proposition \ref{Prop:evolsphere}, we can carry out the first part of the computation of the infinitesimal change of the area of $f$ as we vary the metric only:
\begin{multline} \label{eq:areaestimate1}
 \frac{d}{dt} \Big|_{t = t_0} \area f^*(g(t)) = - \int_{D^2} \tr f^*(\Ric_{t_0}^{\MM(t_0)}) \\
 \leq \frac{3}{4t_0} A(t_0) - \int_{D^2}  \sec^{\MM(t_0)}( df) d {\vol}_{f^* ( g(t_0) )},
\end{multline}
where $\sec^{\MM(t_0)} (df)$ denotes the sectional curvature in the normalized tangential direction of $f$. 
Observe that the last integrand is a continuous function on $D^2$ since the volume form vanishes wherever this tangential sectional curvature is not defined.

In order to avoid issues arising from possible branch points (especially on the boundary of $D^2$), we employ the following trick (inspired by \cite{PerelmanIII}):
Let $\varepsilon > 0$ be a small constant and consider the flat Riemannian disk $(D_\varepsilon = D^2, \varepsilon^2 g_{\textnormal{eucl}})$.
The identity map $h_\varepsilon : D^2 \to D_\varepsilon$ is a conformal and harmonic diffeomorphism and hence the map $f_\varepsilon = (f, h_\varepsilon) : D^2 \to \MM(t_0) \times D_\varepsilon$ is a conformal and harmonic \emph{embedding}.
Denote its image by $\Sigma_\varepsilon = f_\varepsilon(D^2)$.
Since the sectional curvatures on the target manifold $\MM(t_0) \times D_\varepsilon$ arise from pulling back the sectional curvatures on $\MM(t_0)$ via the projection onto the first factor, we have
\begin{equation} \label{eq:areaestimate2}
\lim_{\varepsilon \to 0} \int_{\Sigma_\varepsilon} \sec^{\MM(t_0) \times D_\varepsilon} ( T \Sigma_\varepsilon) d {\vol}_{t_0} =  \int_{D^2}  \sec^{\MM(t_0)}( df) d {\vol}_{f^*(g(t_0))}.
\end{equation}
We can now proceed as in the proof of Proposition \ref{Prop:evolsphere}, using the fact that the interior sectional curvatures of $\Sigma_\varepsilon$ are not larger than the corresponding ambient ones as well as the Theorem of Gau\ss-Bonnet:
\begin{equation} \label{eq:areaestimate3}
 \int_{\Sigma_\varepsilon} \sec^{\MM(t_0) \times D_\varepsilon} ( T \Sigma_\varepsilon) d {\vol}_{t_0} \geq \int_{\Sigma_\varepsilon} \sec^{\Sigma_\varepsilon} ( T \Sigma_\varepsilon) d {\vol}_{t_0} = 2 \pi \chi(\Sigma_\varepsilon) + \int_{\partial \Sigma_\varepsilon} \kappa^{\Sigma_\varepsilon}_{\partial \Sigma_\varepsilon, t_0} d s_{t_0}.
\end{equation}
Here $\kappa^{\Sigma_\varepsilon}_{\partial \Sigma_\varepsilon, t_0}$ denotes the intrinsic geodesic curvature of $\partial \Sigma_{\varepsilon, t_0}$ within $\Sigma_{\varepsilon, t_0}$.
Note that $\chi (\Sigma_\varepsilon) = \chi (D^2) = 1$.

We now estimate the last integral.
Let $\gamma_{t_0, \varepsilon} : S^1(l_{t_0, \varepsilon}) \to \partial \Sigma_{\varepsilon}$, $i = 1, \ldots, m$ be a unit-speed parameterization of the boundary of $\Sigma_\varepsilon$.
Denote by $\gamma^{\MM(t_0)}_{t_0, \varepsilon}(s)$ its component function in $\MM(t_0)$ and by $\gamma^{D_\varepsilon}_{t_0, \varepsilon}(s)$ that in $D_\varepsilon$.
Furthermore, let $\nu_{t_0, \varepsilon}(s)$ be the outward-pointing unit-normal field along $\gamma_{t_0, \varepsilon}(s)$ that is tangent to $\Sigma_{\varepsilon}$.
As before, denote by $\nu^{\MM(t_0)}_{t_0, \varepsilon}(s)$ and $\nu^{D_\varepsilon}_{t_0, \varepsilon} (s)$ the components in the direction of $\MM (t_0)$ and $D_\varepsilon$, respectively.
Note that
\begin{equation} \label{eq:gammasnu22}
 0 = \big\langle  \gamma_{t_0, \varepsilon}' (s), \nu_{t_0, \varepsilon} (s) \big\rangle = \big\langle \big( \gamma_{t_0, \varepsilon}^{\MM(t_0)} \big)' (s) , \nu^{\MM(t_0)}_{t_0, \varepsilon} (s) \big\rangle + \big\langle \big( \gamma_{t_0, \varepsilon}^{D_\varepsilon} \big)' (s), \nu^{D_\varepsilon}_{t_0, \varepsilon} (s) \big\rangle 
\end{equation}
Since $f$ is conformal, we obtain that
\begin{multline*}
  \big\langle \big( \gamma_{t_0, \eps}^{\MM (t_0)} \big)' (s), \nu^{\MM (t_0)}_{t_0, \eps} (s) \big\rangle = \big\langle df \big( \big(\gamma^{D_\eps}_{t_0, \eps} \big)' (s) \big), df \big( \nu^{D_\eps}_{t_0, \eps} (s) \big) \big\rangle \\
  = \lambda_{t_0, \eps} (s) \big\langle \big(\gamma^{D_\eps}_{t_0, \eps} \big)' (s), \nu^{D_\eps}_{t_0, \eps} (s) \big\rangle
\end{multline*}
for some $\lambda_{t_0, \eps} (s) \geq 0$.
So the first summand on the right hand side of (\ref{eq:gammasnu22}) is a non-negative multiple of the second summand.
So these summands cannot have opposite signs and hence
\begin{equation} \label{eq:gammaorthogonaltonu2}  \big\langle \big( \gamma_{t_0, \varepsilon}^{\MM(t_0)} \big)' (s) , \nu^{\MM(t_0)}_{t_0, \varepsilon} (s) \big\rangle = \big\langle \big( \gamma_{t_0, \varepsilon}^{D_\varepsilon} \big)' (s), \nu^{D_\varepsilon}_{t_0, \varepsilon} (s) \big\rangle = 0. \end{equation}
Now note that $\gamma_{t_0, \varepsilon}^{\MM(t_0)}$ is a parameterization of $\gamma_{t_0}$, whose geodesic curvature is bounded by $\Gamma t_0^{-1}$.
Moreover, the geodesic curvature of $\gamma_{t_0, \varepsilon}^{D_\varepsilon}$ is equal to $\varepsilon^{-1}$.
Denote the geodesic curvature of $\gamma_{t_0}$ in $\MM (t_0)$ at any $p \in \gamma_{t_0}$ by $\kappa_{\gamma_{t_0}, t_0} (p)$.
Using (\ref{eq:gammaorthogonaltonu2}), we conclude that
\begin{multline} \label{eq:kappaintwoparts}
 \int_{\partial\Sigma_\varepsilon} \kappa_{\partial \Sigma_\varepsilon, t_0}^{\Sigma_\varepsilon} d s_{t_0} = - \int_0^{l_{t_0, \varepsilon}} \Big\langle \frac{D}{ds} \Big( \frac{d}{ds}  \gamma^{\MM(t_0)}_{t_0, \varepsilon} (s) \Big), \nu^{\MM(t_0)}_{t_0, \varepsilon}(s) \Big\rangle ds \\
 - \int_0^{l_{t_0, \varepsilon}} \Big\langle \frac{D}{ds} \Big( \frac{d}{ds}  \gamma^{D_\varepsilon}_{t_0, \varepsilon} (s) \Big), \nu^{D_\varepsilon}_{t_0, \varepsilon}(s) \Big\rangle ds \\
 = \int_0^{l_{t_0, \varepsilon}} \big\langle \kappa_{\gamma_{t_0}, t_0} (\gamma^{\MM(t_0)}_{t_0, \varepsilon} (s)), \nu_{t_0, \varepsilon}^{\MM (t_0)} (s) \big\rangle  \big| \big( \gamma_{t_0, \varepsilon}^{\MM (t_0)} \big)' (s) \big|^2  ds \\
 + \int_0^{l_{t_0, \varepsilon}} \varepsilon^{-1} \big| \big( \gamma_{t_0, \varepsilon}^{D_\varepsilon} \big)' (s) \big|_{\eps^2 g_{\textnormal{eucl}}}^2 \big| \nu_{t_0, \varepsilon}^{D_\varepsilon} (s) \big|_{\eps^2 g_{\textnormal{eucl}}} ds. 
\end{multline}
As indicated, the norms of the vectors in the last integral are taken with respect to $\eps^2 g_{\textnormal{eucl}}$.

We now analyze the first integral on the right-hand side of (\ref{eq:kappaintwoparts}) by substituting $\gamma_{t_0, \varepsilon}^{\MM (t_0)} : S^1 (l_{t_0, \eps}) \to \gamma_{t_0}$ by a unit speed parameterization $\td\gamma_{t_0} : S^1 (l_{t_0}) \to \gamma_{t_0}$, where $l_{t_0} = \ell (\gamma_{t_0})$.
In doing this, we have to replace $s$ by $\big( \gamma_{t_0, \varepsilon}^{\MM (t_0)} \big)^{-1} ( \td\gamma_{t_0} (s))$, where $\big( \gamma_{t_0, \varepsilon}^{\MM (t_0)} \big)^{-1} : \gamma_{t_0} \to S^1$ denotes the inverse map of $\gamma_{t_0, \varepsilon}^{\MM (t_0)}$.
Moreover, the length element $ds$ changes by a factor of $\big| \big( \gamma_{t_0, \varepsilon}^{\MM (t_0)} \big)' (s) \big|^{-1}$.
So we obtain
\begin{multline} \label{eq:firstintestimate}
  \int_0^{l_{t_0, \varepsilon}} \big\langle \kappa_{\gamma_{t_0}, t_0} (\gamma^{\MM(t_0)}_{t_0, \varepsilon} (s)), \nu_{t_0, \varepsilon}^{\MM (t_0)} (s) \big\rangle  \big| \big( \gamma_{t_0, \varepsilon}^{\MM (t_0)} \big)' (s) \big|^2  ds \\
   = \int_0^{l_{t_0}}  \big\langle \kappa_{\gamma_{t_0}, t_0} ( \td\gamma_{t_0} (s) ), \nu_{t_0, \varepsilon}^{\MM (t_0)} \big( \big( \gamma_{t_0, \varepsilon}^{\MM (t_0)} \big)^{-1} ( \td\gamma_{t_0} (s)) \big) \big\rangle \\ \cdot  
    \big| \big( \gamma_{t_0, \varepsilon}^{\MM(t_0)} \big)' \big( \big( \gamma_{t_0, \varepsilon}^{D_\varepsilon} \big)^{-1} (\td\gamma_{t_0} (s)) \big) \big|ds.
\end{multline}
Next, we estimate the last integral in (\ref{eq:kappaintwoparts}) by substituting the map $\gamma_{t_0, \eps}^{D_\eps} : \partial D^2 \to \partial D^2$ by the identity.
Similarly as before, we have to replace $s$ by $\big( \gamma_{t_0, \eps}^{D_\eps} \big)^{-1} (s)$, where $\big( \gamma_{t_0, \eps}^{D_\eps} \big)^{-1} : \partial D^2 \to \partial D^2$ denotes the inverse map of $\gamma_{t_0, \eps}^{D_\eps}$, and change $ds$ by a factor of $\big| \big( \gamma_{t_0, \eps}^{D_\eps} \big)' (s)|^{-1}_{g_\textnormal{eucl}}$.
Additionally, using the fact that the norm of $\nu_{t_0, \varepsilon}^{D_\varepsilon} (s)$ with respect to the metric $\eps^2 g_{\textnormal{eucl}}$ is bounded by $1$, we obtain
\begin{multline} \label{eq:secondintestimate}
 \int_0^{l_{t_0, \varepsilon}} \varepsilon^{-1} \big| \big( \gamma_{t_0, \varepsilon}^{D_\varepsilon} \big)' (s) \big|_{\eps^2 g_{\textnormal{eucl}}}^2 \big| \nu_{t_0, \varepsilon}^{D_\varepsilon} (s) \big|_{\eps^2 g_{\textnormal{eucl}}} ds \leq \eps \int_0^{l_{t_0, \eps}} \big| \big( \gamma_{t_0, \varepsilon}^{D_\varepsilon} \big)' (s) \big|_{ g_{\textnormal{eucl}}}^2 ds  \\
 = \int_{\partial D^2} \big| \big( \gamma_{t_0, \varepsilon}^{D_\varepsilon} \big)' \big( \big( \gamma_{t_0, \varepsilon}^{D_\varepsilon} \big)^{-1} (s) \big) \big|_{\eps^2 g_{\textnormal{eucl}}} ds.
\end{multline}
Now we let $\varepsilon \to 0$.
Observe that away from the branch points of $f$
\[ \lim_{\varepsilon \to 0} \nu_{t_0, \varepsilon}^{\MM (t_0)} \big( \big( \gamma_{t_0, \varepsilon}^{\MM (t_0)} \big)^{-1} (s) \big) = \nu^f_{t_0} (s), \]
where $\nu_{t_0}^f$ is the outward-pointing unit normal vector field along $\gamma_{t_0}$ that is tangential to $f$.
Moreover, away from the branch points of $f$
\[ \lim_{\varepsilon \to 0} \big| \big( \gamma_{t_0, \varepsilon}^{\MM(t_0)} \big)' \big( \big( \gamma_{t_0, \varepsilon}^{D_\varepsilon} \big)^{-1} (s) \big) \big|= 1, \qquad  \lim_{\varepsilon \to 0}  \big| \big( \gamma_{t_0, \varepsilon}^{D_\varepsilon} \big)' \big( \big( \gamma_{t_0, \varepsilon}^{D_\varepsilon} \big)^{-1} (s) \big) \big|_{\eps^2 g_{\textnormal{eucl}}} = 0
\]
So, using (\ref{eq:kappaintwoparts}), (\ref{eq:firstintestimate}) and (\ref{eq:secondintestimate}) gives us
\begin{equation} \label{eq:boundaryintegralsexample}
  \lim_{\varepsilon \to 0} \int_{\partial\Sigma_\varepsilon} \kappa_{\partial \Sigma_\varepsilon, t_0}^{\Sigma_\varepsilon}  d s_{t_0} = \int_{\gamma_{t_0}} \big\langle \kappa_{\gamma_{t_0}, t_0} (s), \nu_{t_0}^f (s) \big\rangle ds_{t_0} \geq - \Gamma t_0^{-1/2} \ell (\gamma_{t_0}) > - a \Gamma.
\end{equation}
Combining this with (\ref{eq:areaestimate1}), (\ref{eq:areaestimate2}) and (\ref{eq:areaestimate3}) yields
\[ \frac{d}{dt} \Big|_{t = t_0} \area f^* ( g(t)) \leq \frac{3}{4 t_0} A(t_0) - 2 \pi  + a \Gamma . \]
So in the barrier sense
\[ \frac{d}{dt^+} \Big|_{t = t_0} A(t) \leq \frac{3}{4t_0} A(t_0) - 2\pi + a \Gamma. \]
Thus
\[
 \frac{d}{dt^+} \big[ t^{1/4} \big( t^{-1} A(t) + 4 (2\pi  -  a \Gamma ) \big) \big] \leq 0.
\]

Analogously as in the proof of Proposition \ref{Prop:evolsphere}, we conclude that $A(t)$ is lower semi-continuous.
The desired monotonicity follows now immediately.
The bound on $T_2$ follows again from the fact that $A(T_2) > 0$.
\end{proof}

\section{Simplicial complexes} \label{sec:Defpolygonalcomplex}
We briefly recall the notion of simplicial complexes, which will be used throughout the whole paper.
Note that in the following we will only be interested in simplicial complexes that are $2$ dimensional, pure and locally finite.
For brevity we will always implicitly assume these properties when referring to the term ``simplicial complex''.

\begin{Definition}[simplicial complex] \label{Def:simplcomplex}
A \emph{$2$-dimensional simplicial complex} $V$ is a topological space that is the union of embedded, closed $2$-simplices (triangles), $1$-simplices (intervals) and $0$-simplices (points) such that any two distinct simplices are either disjoint or their intersection is equal to another simplex whose dimension is strictly smaller than the maximal dimension of both simplices.
$V$ is called \emph{finite} if the number of these simplices is finite.

In this paper, we assume $V$ moreover to be \emph{locally finite} and \emph{pure}.
The first property demands that every simplex of $V$ is contained in only finitely many other simplices and the second property states that every $0$ or $1$-dimensional simplex is contained in a $2$-simplex.
We will also assume that all $2$ and $1$-simplices are equipped with differentiable parameterizations that are compatible with respect to restriction.

We will often refer to the $2$-simplices of $V$ as \emph{faces}, the $1$-simplices as \emph{edges} and the $0$-simplices as \emph{vertices}.
The \emph{$1$-skeleton} $V^{(1)}$ is the union of all edges and the \emph{$0$-skeleton} $V^{(0)}$ is the union of all vertices of $V$.
The \emph{valency} of an edge $E \subset V^{(1)}$ denotes the number of adjacent faces, i.e. the number of $2$-simplices that contain $E$.
The \emph{boundary} $\partial V$ is the union of all edges of valency $1$.
\end{Definition}

We will also use the following notion for maps from simplicial complexes into manifolds.

\begin{Definition}[piecewise smooth map]
Let $V$ be a simplicial complex, $M$ an arbitrary differentiable manifold (not necessarily $3$-dimensional) and $f : V \to M$ a continuous map.
We call $f$ \emph{piecewise smooth} if $f$ restricted to the interior of each face of $V$ is smooth and bounded in $W^{1,2}$ and if $f$ restricted to each edge $E \subset V^{(1)}$ is smooth away from finitely many points.
\end{Definition}

Given a Riemannian metric $g$ on $M$ and a sufficiently regular map $f : V \to M$ (e.g. piecewise smooth) we define its area, $\area (f)$, to be the sum of $\area (f|_{\Int F})$ over all faces $F \subset V$ and the length of the $1$-skeleton $\ell (f |_{V^{(1)}})$ to be the sum of the length $\ell (f|_E)$ over all edges $E \subset V^{(1)}$.

\section{Existence of minimizers of simplicial complexes} \label{sec:existenceofvminimizers}
\subsection{Introduction and overview}
Let in this section $(M, g)$ always be a compact Riemannian manifold (not necessarily $3$ dimensional) with $\pi_2 (M) = 0$.
We will also fix the following notation: for every continuous contractible loop $\gamma : S^1 \to M$ we denote by $A(\gamma)$ the infimum over the areas of all continuous maps $f : D^2 \to M$ that are continuously differentiable on the interior of $D^2$, bounded in $W^{1,2}$ and for which $f |_{\partial D^2} = \gamma$.

Consider a finite simplicial complex $V$ as well as a continuous map $f_0 : V \to M$ such that $f_0 |_{\partial V}$ is a smooth embedding.
The goal of this section is motivated by the question of finding an area-minimizer within the same homotopy class of $f_0$, i.e. a map $f : V \to M$ that is homotopic to $f_0 : V \to M$ relative to $\partial V$ and whose area is equal to 
\[ A(f_0) : = \inf \big\{ \area f' \;\; : \;\; f' \simeq f_0 \;\; \text{relative to $\partial V$} \big\}. \]
(Here the maps $f' : V \to M$ are assumed to be continuous and continuously differentiable when restricted to $V \setminus V^{(1)}$ and $V^{(1)}$ as well as bounded in $W^{1,2}$ when restricted to each face of $V$.)
This problem, however, seems to be very difficult, since it is not clear how to control e.g. the length of the $1$-skeleta of a sequence of minimizers.

To get around these analytical issues, we instead seek to minimize the quantity $\area (f) + \ell (f |_{V^{(1)}})$.
Here $ \ell (f |_{V^{(1)}})$ denotes the sum of the lengths of all edges of $V$ under $f$.
It will turn out that this change has no negative effect when we apply our results to the Ricci flow in section \ref{sec:areaevolutionunderRF}.
To summarize, we are looking for maps $f : V \to M$ that are homotopic to $f_0$ relative $\partial V$ and for which $\area (f) + \ell (f |_{V^{(1)}})$ is equal (or close) to
\[ A^{(1)} (f_0) := \inf \big\{ \area (f') + \ell( f' |_{V^1} )  \;\; : \;\; f' \simeq f_0 \;\; \text{relative to $\partial V$} \big\}. \]
We will be able to show that such a minimizer exists in a certain sense.
More specifically, we will find a map $f : V^{(1)} \to M$ of regularity $C^{1,1}$ on the $1$-skeleton that can be extended onto $V$ to a minimizing sequence for $A^{(1)}$.
This implies that the sum of $A(f|_{\partial F})$ over all faces $F \subset V$ plus $\ell (f)$ is equal to $A^{(1)}(f_0)$.
So the existence problem for $f$ is reduced to solving the Plateau problem for each loop $f |_{\partial F}$.
The only difficulty that we may encounter then is that $f |_{\partial F}$ might a priori have (finitely or infinitely many) self-intersections.
Unfortunately, taking this possibility into account makes several arguments quite tedious and might obscure the main idea in a forest of details.

The second goal of this section (see subsection \ref{subsec:1skeletonstruc}) is to understand the geometry of a minimizer along the $1$-skeleton.
In the case in which $f : V^{(1)} \to M$ is injective, our findings can be presented as follows.
In this case we can solve the Plateau problem for the loop $f |_{\partial F}$ for each face $F \subset V$ and extend $f : V^{(1)} \to M$ to a map $f : V \to M$ that is smooth on $V \setminus V^{(1)}$ and $C^{1,1}$ on $V^{(1)}$ and $C^{1, \alpha}$ on every (closed) face away from the vertices.
Consider and edge $E \subset V^{(1)} \setminus \partial V$ of valency $v_E$ and denote by $\kappa : E \to TM$ the geodesic curvature (defined almost everywhere) of $f |_E$ and let $\nu^{(1)}_E, \ldots, \nu^{(v_E)}_E : E \to TM$ be unit vector fields that are normal to $f |_E$ and outward pointing tangential to $f$ restricted to those faces $F \subset V$ that are adjacent to $E$.
A simple variational argument will then yield the identities
\begin{equation} \label{eq:simplenuiskappa}
 \nu^{(1)}_E +  \ldots + \nu^{(v_E)}_E = \kappa_E \qquad \text{and} \qquad \big\langle \nu^{(1)}_E +  \ldots + \nu^{(v_E)}_E, \kappa_E  \big\rangle \geq 0.
\end{equation}
This set of equalities and inequalities is the second main result of this section and some time is spent on expressing these identities in the case in which the loops $f |_{\partial F}$ are allowed to have self-intersections.
We remark that in the case in which $f |_{V^{(1)}}$ is injective this equality and a bootstrap argument can be used to show that $f$ is actually smooth on each (closed) face away from $V^{(0)}$.

Observe that in general it might happen that two or more edges are mapped to the same segment under $f$ (this could also happen for subsegments of these edges or for subsegments of one and the same edge).
It would then become necessary to take the sum over all faces that are adjacent to either of these edges on the left hand side of (\ref{eq:simplenuiskappa}) and a multiple of $\kappa_E$ on the right hand side of the equation in (\ref{eq:simplenuiskappa}).
These combinatorics become even more involved by the fact that, at least a priori, $f |_{\partial F}$ can for example intersect in a subset of empty interior but positive measure.

All important results of this section will be summarized in Proposition \ref{Prop:existenceofVminimizer}.

\subsection{Construction and regularity of the map on the 1-skeleton} \label{subsec:regularityon1skeleton}
Consider again the given continuous map $f_0 : V \to M$ for which $f_0 |_{\partial V}$ is a smooth embedding and let $f_1, f_2, \ldots : V \to M$ be a minimizing sequence for $A^{(1)}(f_0)$.
More specifically, we want each $f_k$ to be continuous and homotopic to $f_0$ relative $\partial V$, continuously differentiable when restricted to $V \setminus V^{(1)}$ and $V^{(1)}$ as well as bounded in $W^{1,2}$ when restricted to each face and
\[ \lim_{k \to \infty} \big( \area (f_k) + \ell (f_k |_{V^{(1)}} ) \big) = A^{(1)}(f_0). \]
By compactness of $M$ we may assume that, after passing to a subsequence, $f_k |_{V^{(0)}}$ converges pointwise.
Next, observe that every edge $E \subset V^{(1)}$ is equipped with a standard parameterization by an interval $[0,1]$ (see Definition \ref{Def:simplcomplex}).
We can then reparameterize each $f_k$ such that for every edge $E \subset V^{(1)}$ the restriction $f_k |_E$ is parameterized by constant speed.
Since $\ell ( f_k |_E )$ is uniformly bounded, we can pass to another subsequence such that $f_K |_E$ converges uniformly.
So we may assume that $f_k |_{V^{(1)}}$ converges uniformly to a Lipschitz map $f : V^{(1)} \to M$ and that $\ell (f |_{V^{(1)}}) \leq \liminf_{k \to \infty} \ell (f_k |_{V^{(1)}})$.
It is our first goal to derive regularity results for $f$.
Before doing this we characterize the map $f$, so that we can forget about the sequence $f_k$.

\begin{Lemma} \label{Lem:existenceon1skeleton}
The map $f$ is homotopic to $f_0 |_{V^{(1)}}$ relative to $\partial V$ and is parameterized by constant speed and if $F_1, \ldots, F_n$ are the faces of $V$, then
\[ A( f |_{\partial F_1} ) + \ldots + A (f |_{\partial F_n}) + \ell(f) = A^{(1)} (f_0). \]
Moreover, for every continuous map $f' : V^{(1)} \to M$ that is homotopic to $f_0 |_{V^{(1)}}$ relative to $\partial V$ we have
\[ A( f' |_{\partial F_1} ) + \ldots + A (f' |_{\partial F_n}) + \ell(f') \geq A^{(1)} (f_0). \]
\end{Lemma}

\begin{proof}
The fact that $f$ is homotopic to $f_0 |_{V^{(1)}}$ relative to $\partial V$ follows from the uniform convergence.

For every face $F_j$ consider the boundary loop $f |_{\partial F_j} : \partial F_j \approx S^1 \to M$ which is a Lipschitz map.
Recall that the loops $f_k |_{\partial F_j}$ converge uniformly to $f |_{\partial F_j}$.
So using the exponential map and assuming that $k$ is large enough, we can find a homotopy $H_k : \partial F_j \times [0,1] \to M$ between $f_k |_{\partial F_j}$ and $f |_{\partial F_j}$ that is Lipschitz on $\partial F_j \times [0,1]$ and smooth on $\partial F_j \times (0,1)$ and whose area goes to $0$ as $k \to \infty$.
Gluing $H_k$ together with $f_k |_{F_j} : F_j \to M$ and mollifying around the seam yields a continuous map $f^*_{j,k} : F_j \to M$ that is smooth on $\Int F_j$ such that $f^*_{j,k} |_{\partial F_j} = f |_{\partial F_j}$ and such that $\area f^*_{j,k} - \area f_k |_{F_j}$ goes to $0$ as $k \to \infty$ (here we are using the fact that $f_k |_{F_j}$ is bounded in $W^{1,2}$).
Hence $A( f |_{\partial F_j} ) \leq \liminf_{k \to \infty} \area f_k |_{F_j}$ and we obtain
\begin{multline*}
 A( f |_{\partial F_1} ) + \ldots + A (f |_{\partial F_n}) + \ell(f) \\ \leq \liminf_{k \to \infty} \big( \area (f_k |_{\partial F_1}) + \ldots + \area (f_k |_{\partial F_n}) + \ell ( f_k |_{V^{(1)}} ) \big) = A^{(1)} (f_0).
\end{multline*}
For the reverse inequality it remains to establish the last statement of the claim.
This will then also imply that $\lim_{k \to \infty} \ell (f_k |_{\partial V^{(1)}} ) = \ell (f)$ and hence that $f$ is parameterized by constant speed.

Consider a continuous and rectifiable map $f' : V^{(1)} \to M$ that is homotopic to $f_0 |_{V^{(1)}}$ relative to $\partial V$.
We can find smoothings $f'_k : V^{(1)} \to M$ of $f'$ such that $f'_k$ converges uniformly to $f'$ and $\lim_{k \to \infty} \ell (f'_k) = \ell (f')$.
Now for every face $F_j$, we can again find a homotopy $H'_{j, k} : \partial F_j \times [0,1] \to M$ of small area between $f' |_{\partial F_j}$ and $f'_k |_{\partial F_j}$ and by another gluing argument, we can construct continuous maps $f''_{j,k} : F_j \to M$ with $f''_{j, k} |_{\partial F_j} = f'_k |_{\partial F_j}$ that are smooth on $\Int F_j$ such that $\lim_{k \to \infty} \area f''_{j,k} = A ( f' |_{\partial F_j} )$.
Hence, we can extend each $f'_k : V^{(1)} \to M$ to a map $f''_k : V \to M$ of the right regularity such that 
\[ A(f' |_{\partial F_1}) + \ldots + A(f' |_{\partial F_n}) + \ell (f') =  \lim_{k \to \infty} \big( \area(f''_k) + \ell (f''_k |_{V^{(1)}}) \big) \geq A^{(1)} (f_0) . \]
This proves the desired result.
\end{proof}

We also need the following isoperimetric inequality.

\begin{Lemma} \label{Lem:isoperimetricdistancetoaxis}
Let $\gamma : S^1 \to \IR^n$ be a rectifiable loop such that $\gamma$ restricted to the lower semicircle of $S^1$ parameterizes an interval on the $x_1$-axis $x_2 = \ldots = x_n = 0$ and $\gamma$ restricted to the upper semicircle has length $l$.
Denote by $a$ the maximum of the euclidean norm of the $(x_2, \ldots, x_n)$ component of all points on $\gamma$ (i.e. the maximal distance from the $x_1$-axis).
Then $A(\gamma) \leq la$.
\end{Lemma}

\begin{proof}
Let $\ov\gamma : [0,l] \to \IR^n$ be a parameterization by arclength of $\gamma$ restricted to the upper semicircle of $S^1$.
Let $0 = s_0 < s_2 < \ldots < s_m = l$ be a subdivision of the interval $[0,l]$.
Let $y_i$ be the $x_1$-coordinate of $\ov\gamma(s_i)$ and $\sigma_i$ a straight segment between $\ov\gamma(s_i)$ and $(y_i, 0, \ldots, 0)$ for each $i = 0, \ldots, m$.
For each $i = 1, \ldots, m$ let $\ov\gamma_i$ be the loop that consists of $\ov\gamma |_{[s_{i-1}, s_i]}, \sigma_{i-1}, \sigma_i$ and the interval between $(y_{i-1}, 0, \ldots, 0), (y_i, 0, \ldots, 0)$.
We set $A^* (s_0, \ldots, s_m) = A(\gamma_1) +  \ldots + A( \gamma_m)$.

Let $i \in \{ 1, \ldots, m-1 \}$.
We claim that if we remove $s_i$ from the list of subdivisions, then the value of $A^* (s_0, \ldots, s_m)$ does not increase.
In fact, if $y_{i-1} \leq y_i \leq y_{i+1}$ or $y_{i-1} \geq y_i \geq y_{i+1}$, then this is claim is true since any two maps $h_i, h_{i+1} : D^2 \to M$ that restrict to $\gamma_i, \gamma_{i+1}$ on $S^1$ can be glued together along $\sigma_i$.
On the other hand, if $y_{i-1} \leq y_{i+1} \leq y_i$, then $h_i, h_{i+1}$ can be glued together along the union of $\sigma_i$ with the interval between $(y_{i+1}, 0, \ldots, 0), (y_i, 0, \ldots, 0)$.
The other cases follow analogously.
Multiple application of this finding yields $A(\gamma) \leq A^* (s_0, \ldots, s_m)$.

Let now $\gamma'_i$ be the loop that consists of the straight segment between  $\ov\gamma(s_{i-1}), \linebreak[1] \ov\gamma(s_i)$, the segments $\sigma_{i-1}, \sigma_i$ and the interval between $(y_{i-1}, 0, \ldots, 0), (y_i, 0, \ldots, 0)$.
Moreover, let $\gamma''_i$ be the loop that consists of the straight segment between $\ov\gamma(s_{i-1}), \linebreak[1] \ov\gamma(s_i)$ and the curve $\ov\gamma|_{[s_{i-1}, s_i]}$.
Then by the isoperimetric inequality and some basic geometry
\[ A(\gamma_i) \leq A(\gamma'_i) + A(\gamma''_i) \leq a \ell (\ov\gamma|_{[s_{i-1}, s_i]}) + C (\ell (\ov\gamma|_{[s_{i-1}, s_i]}))^2. \]
Adding up this inequality for all $i = 1, \ldots, m$ yields
\[ A(\gamma) \leq A^* (s_0, \ldots, s_m) \leq a l + \sum_{i=1}^m C (\ell (\ov\gamma|_{[s_{i-1}, s_i]}))^2. \]
The right hand side converges to $0$ as the mesh size of the subdivisions approaches zero.
\end{proof}

The following Lemma is our main regularity result.

\begin{Lemma} \label{Lem:regularityon1skeleton}
The map $f : V^{(1)} \to M$ has regularity $C^{1,1}$ on every edge $E \subset V^{(1)}$.
\end{Lemma}

\begin{proof}
Let $E \subset V^1$ be an edge and equip $E$ with a smooth parameterization of an interval such that $f|_E$ is parameterized by constant speed.
We now establish the regularity of the map $f_E = f |_E : E \to M$ up to the endpoints of $E$.
Assume $\ell (f|_E) > 0$, since otherwise we are done.
After scaling the interval by which $E$ is parameterized, we may assume without loss of generality that $f_E$ is parameterized by arclength, i.e. that
\[ \ell (f_E |_{[s_1, s_2]}) = s_2 - s_1 \qquad \text{for every interval} \qquad [s_1, s_2] \subset E. \]

Let $\varepsilon > 0$ be smaller than the injectivity radius of $M$ and observe that whenever we choose exponential coordinates $(y_1, \ldots, y_n)$ around a point $p \in M$ then under these coordinates we have the following comparison with the Euclidean metric $g_{\textnormal{eucl}}$:
\begin{equation} \label{eq:gminusgeuclforregularity}
 | g - g_{\textnormal{eucl}} | < C_1 r^2
\end{equation}
for some uniform constant $C_1$ (here $r$ denotes the radial distance from $p$).
Assume moreover that $\varepsilon$ is chosen small enough such that $g$ is $2$-Bilipschitz to $g_{\textnormal{eucl}}$.

Consider three parameters $s_1, s_2, s_3 \in E$ such that $s_1 < s_2 < s_3 < s_1 + \frac1{10} \varepsilon$.
We set $x_i = f_E(s_i)$, $l = |s_3 - s_1| = \ell( f_E |_{[s_1, s_3]})$ as well as $d = \dist(x_1, x_3)$ and we denote by $\gamma$ a minimizing geodesic segment between $x_1$ and $x_3$.
Consider now the competitor map $f'$ that agrees with $f$ on $(V^{(1)} \setminus E) \cup (E \setminus (s_1, s_2))$ and that maps the interval $[s_1, s_3]$ to the segment $\gamma$.

Let us first bound the area gain for such a competitor.
Denote by $\gamma^* : S^1 \to M$ the loop that consists of the curves $f_E |_{[s_1, s_3]}$ and $\gamma$.
Choose geodesic coordinates $(y_1, \ldots, y_n)$ around $x_1$ such that $\gamma$ can be parameterized by $(t, 0, \ldots, 0)$ and denote by $a$ the maximum of the euclidean norm of the $(y_2, \ldots, y_n)$-component of $f_E$ on $[s_1, s_3]$.
By Lemma \ref{Lem:isoperimetricdistancetoaxis} we have
\[ A(\gamma^*)  \leq 8 l a. \]
(Recall that $g$ is $2$-Bilipschitz to the euclidean metric.)
Let $F_1, \ldots, F_v$ be the faces that are adjacent to $E$.
Then for each $j = 1, \ldots, v$ we have
\[ A( f' |_{\partial F_j} ) \leq A( f |_{\partial F_j} ) + A(\gamma^*) \leq A( f|_{\partial F_j}) + 8la. \]
Moreover, $\ell (f') \leq \ell(f) - l + d$.
So by the inequality of Lemma \ref{Lem:existenceon1skeleton} we obtain
\begin{equation} \label{eq:lminusd}
 l - d \leq 8v \cdot l a .
\end{equation}

Let now $l'$ be the length of the segment parameterized by $f_E |_{[s_1, s_3]}$ with respect to the euclidean metric $g_{\textnormal{eucl}}$ in the coordinate system $(y_1, \ldots, y_n)$.
Then $\tfrac12 l' \leq l \leq 2l'$.
Moreover, we obtain the following improved bound on $l'$ using (\ref{eq:gminusgeuclforregularity}):
\begin{multline*}
 l = \int_{s_1}^{s_3} \sqrt{g(f'_E (s), f'_E(s))} ds \geq \int_{s_1}^{s_3} \sqrt{ (1 - C_1 (l')^2) g_{\textnormal{eucl}} (f'_E(s), f'_E(s))} ds \\
  \geq \sqrt{ 1- 4C_1 l^2} \; l'.
\end{multline*}
By basic trigonometric estimates with respect to the euclidean metric in the coordinate system $(y_1, \ldots, y_n)$ we obtain
\[ d^2 + 4 a^2 \leq (l')^2  . \]
So
\begin{equation} \label{eq:dandarelationprecise}
  (1-4C_1 l^2) (d^2 + 4 a^2) \leq l^2.
\end{equation}
Plugging in (\ref{eq:lminusd}) yields with $c = \frac14 v^{-2}$
\[ (1-4C_1 l^2) ( l^2 d^2 + c (l-d)^2) \leq  l^4. \]
And hence for $l < \frac14 C_1^{-1/2}$
\[ \tfrac{c}2 (l-d)^2 \leq l^2 (l-d)(l+d) + 4 C_1 l^4 d^2 \leq 2 l^3 (l-d) + 4 C_1 l^6. \]
This inequality implies that if $l-d \geq l^3$, then $\tfrac{c}2 (l-d) \leq 2 l^3 + 4 C_1 l^3$.
So in either case (if $l-d \geq l^3$ or if $l-d < l^3$) there is a universal constant $C_2$ such that
\begin{equation} \label{eq:lminusdcubebound}
 l - d \leq C_2 l^3.
\end{equation}
In particular, if $l$ is smaller than some uniform constant, then
\[ \tfrac12 d \leq l \leq 2d. \]
We will in the following always assume that this bound holds whenever we compare the intrinsic and extrinsic distance between two close points on $f_E$.

Next, we plug (\ref{eq:lminusdcubebound}) back into (\ref{eq:dandarelationprecise}) and obtain a bound on $a$ for small $l$:
\[ a \leq \sqrt{\frac{(l-d) (l+d) + 4 C_1 l^2 d^2}{4 ( 1- 4 C_1 l^2)}} \leq \sqrt{ C_2 l^3 \cdot 2 l + 4 C_1 l^2 d^2} \leq C_3 l^2 \]
for some uniform constant $C_3$.
Now consider the point $x_2$ on $f_E ([s_1, s_3])$, set $l_1 = \ell (f_E |_{[s_1,s_2]})$ and let $\alpha \geq 0$ be the angle between the geodesic segment $\gamma$ from $x_1$ to $x_3$ and the geodesic segment $\gamma_1$ from $x_1$ to $x_2$.
Observe that the angle $\alpha$ between $\gamma$ and $\gamma_1$ is the same with respect to both $g$ and $g_{\textnormal{eucl}}$.
Moreover, by our previous conclusion applied to $x_1, x_2$ instead of $x_1, x_3$, the length of $\gamma_1$ is bounded from below by $\frac12 l_1$.
So by basic trigonometry we find that there are uniform constants $\varepsilon_0 > 0$ and $C_4 < \infty$ such that we have
\begin{equation} \label{eq:angleboundforregularity}
 \alpha \leq C_4 l \qquad \text{if} \qquad l_1 \geq \tfrac12 l \quad \text{and} \quad l < \varepsilon_0 .
\end{equation}

We can now establish the differentiability of $f_E$.
Let $s, s', s'' \in E$ such that $s < s' < s'' < s + \varepsilon_0$, set $x = f_E(s)$, $x' = f_E (s')$, $x'' = f_E (s'')$ and choose minimizing geodesic segments $\gamma', \gamma''$ between $x, x'$ and $x, x''$.
Let $\alpha \geq 0$ bet the angle between $\gamma', \gamma''$ at $x$.
For each $i \geq 1$ for which $s + 2^{-i} \in E$ we set $x_i = f_E ( s+ 2^{-i})$ and we choose a minimizing geodesic segment $\gamma_i$ between $x$ and $x_i$.
Choose moreover indices $i' \geq i'' \geq 1$ such that $2^{- i'} \leq s' - s< 2^{- i' + 1}$ and $2^{- i''} \leq s'' - s< 2^{- i'' + 1}$.
Then by (\ref{eq:angleboundforregularity})
\begin{multline*}
\alpha \leq \nangle_x (\gamma'', \gamma_{i''}) + \nangle_x (\gamma_{i''}, \gamma_{i'' + 1}) + \ldots + \nangle_x (\gamma_{i' - 2}, \gamma_{i' - 1}) + \nangle_x (\gamma_{i' - 1}, \gamma') \\
\leq C_4 (s''-s) + C_4 2^{-i''} + C_4 2^{-i'' - 1} + \ldots \\
 \leq C_4 (s'' - s) + 2 C_4 2^{-i''} \leq 3 C_4 (s'' - s).
\end{multline*}
Note also that by (\ref{eq:lminusdcubebound}) the quotients $\frac{\ell(\gamma')}{s'-s}$ and $\frac{\ell(\gamma'')}{s''-s}$ converge to $1$ as $s'' \to s$.
Altogether, this shows that the right-derivative of $f_E$ exists, has unit length and that
\begin{equation} \label{eq:anglebetweenfsandgamma}
 \nangle_x \big( \tfrac{d}{ds^+} f_E (s), \gamma'' \big) \leq 3 C_4  (s'' - s).
\end{equation}
The existence of the left-derivative together with the analogous inequality follows in the same way.
In order to show that the right and left-derivatives agree in the interior of $E$, it suffices to show for any $s \in \Int E$, that the angle between the geodesic segments between $f_E (s), f_E(s-s')$ and $f_E (s), f_E (s+s')$ goes to $\pi$ as $s' \to 0$.
This follows immediately from (\ref{eq:angleboundforregularity}) and the fact that the sum of the angles of small triangles in $M$ goes to $\pi$ as the circumference goes to $0$.

Finally, we establish the Lipschitz continuity of the derivative $f'_E(s)$.
Let $s_1, s_3 \in E$ such that $s_1 < s_3 < s_1 + \varepsilon_0$ and let $s_2 = \frac12 (s_1 + s_3)$ be the midpoint on $f_E$.
Let $\gamma$ and $\gamma_1$ be defined as before and let $\gamma_3$ be the geodesic segment between $x_2 = f_E(s_2)$ and $x_3 = f_E (s_3)$.
Using (\ref{eq:gminusgeuclforregularity}) we find that if we choose geodesic coordinates around $x_1$ or $x_3$, then we can compare angles at different points on $f_E ([s_1, s_3])$ up to an error of $O(|s_3 - s_1|^2)$.
So we can estimate using (\ref{eq:angleboundforregularity}) and (\ref{eq:anglebetweenfsandgamma})
\begin{multline*}
 \nangle ( f'_E(s_1), f'_E(s_3)) \leq \nangle (f'_E (s_1), \gamma_1) + \nangle (\gamma_1, \gamma) \\ + \nangle (\gamma, \gamma_3) + \nangle (\gamma_3, f'_E(s_3)) + O(|s_3 - s_1|^2) \\ \leq 3 C_4 |s_2 - s_1| + 2 C_4 |s_3 - s_1| + 3 C_4 |s_3 - s_2| + O(|s_3 - s_1|^2)  \leq C_5 |s_3 - s_1|
\end{multline*}
for some uniform constant $C_5$.
This finishes the proof.
\end{proof}

Now if for every face $F \subset V$ the map $f |_{\partial F}$ is injective (i.e. an embedding in a proper parameterization), then by solving the Plateau problem for each face (cf \cite{Mor}) we obtain an extension $\td{f} : V \to M$ of $f$ that is homotopic to $f_0$ and for which $\area \td{f} + \ell ( \td{f} |_{V^{(1)}} ) = A^{(1)} (f_0)$.
So in this case the existence of the minimizer is ensured.
In general, however, we need take into account the possibility that $f |_{\partial F}$ has self-intersections.
Note that there might be infinitely many such self-intersections and the set of self-intersections might even have positive $1$ dimensional Hausdorff measure.
This adds some technicalities to our discussion.

\subsection{Results on self-intersections and the Plateau problem}
The following Lemma states that two intersecting curves agree up to order $2$ almost everywhere on their set of intersection.

\begin{Lemma} \label{Lem:speedandcurvatureagreeonintersection}
Let $\gamma : [0, l] \to M$ be a curve of regularity $C^{1,1}$ that is parameterized by arclength.
Then the geodesic curvature along $\gamma$ is defined almost everywhere, i.e. there is a vector field $\kappa : [0,l] \to TM$ along $\gamma$ (i.e. $\kappa (s) \in T_{\gamma(s)} M$ for all $s \in [0,l]$) and a null set $N \subset [0,l]$ such that at each $s \in [0,l] \setminus N$ the curve $\gamma$ is twice differentiable and the geodesic curvature at $s$ equals $\kappa (s)$.

Consider now two such curves $\gamma_1 : [0,l_1] \to M$, $\gamma_2 : [0,l_2] \to M$ with geodesic curvature vector fields $\kappa_1, \kappa_2$.
Assume additionally that $\gamma_1, \gamma_2$ are injective embeddings that are contained in a coordinate chart $(U, (x_1, \ldots, x_n))$ in such a way that there is a vector $v \in \IR^n$ with the property that $\langle \gamma'_i (s), v \rangle \neq 0$ with respect to the euclidean metric for all $s \in [0, l_i]$ and $i = 1,2$.

Let $X_1 = \{ s \in [0,l_1] \;\; : \;\; \gamma_1(s) \in \gamma_2([0,l_2]) \}$ and $X_2 = \{ s \in [0,l_2] \;\; : \;\; \gamma_2(s) \in \gamma_1([0,l_1]) \}$ be the parameter sets of self-intersections.
Then there is a continuously differentiable map $\varphi : [0,l_1] \to \IR$ whose derivative  vanishes nowhere such that $\varphi(X_1) = X_2$ and such that $\gamma_1(s) = \gamma_2 (\varphi(s))$ whenever $s \in X_1$.
Moreover, there are null sets $N_i \subset X_i$ such that $\varphi(N_1) = N_2$ and such that for all $s \in X_1 \setminus N_1$ we have $\varphi'(s) = \pm 1$, $\gamma'_1 (s) = \gamma'_2 (\varphi(s)) \varphi'(s)$ and $\kappa_1 (s) = \kappa_2 (\varphi(s))$.
\end{Lemma}

\begin{proof}
The first statement follows from the fact that a Lipschitz function is differentiable almost everywhere.
Observe that the geodesic curvature can be computed in terms of the first and second derivative of the curve in a local coordinate system.

Let $\varphi : [0,l_1] \to \IR$ be the composition of the projection $s \mapsto \langle \gamma_1(s), v \rangle_{\IR^n}$ with the inverse of the projection $s \mapsto \langle \gamma_2(s), v \rangle_{\IR^n}$  (the scalar product $\langle \cdot, \cdot \rangle_{\IR^n}$ is taken in the coordinates $(x_1, \ldots, x_n)$).
Then by definition $\varphi(X_1) = X_2$ and $\gamma_1(s) = \gamma_2 (\varphi(s))$ whenever $s \in X_1$.
Moreover, $\varphi'(s) \neq 0$ for all $s \in [0, l_1]$.

Next, let $N'_i \subset [0,l_i]$ be the null sets from the first part outside of which $\kappa_i$ is equal to the geodesic curvature of $\gamma_i$.
Let moreover, $N^*_1 \subset X_1$ be the set of isolated points of $X_i$.
Note that $N^*_1$ is a null set.
We now claim that the Lemma holds for $N_1 = X_1 \cap (N'_1 \cup \varphi^{-1} (N'_2) \cup N^*_1 )$ and $N_2 = X_2 \cap ( \varphi(N_1) \cup N'_2 )$.
The sets $N_1, N_2$ are null sets.
Let now $s \in X_1 \setminus N_1$.
Observe that for $s'$ close to $s$, we have
\[ \gamma_1 (s') = \gamma_1 (s) + (s'-s) \gamma'_1 (s) + \tfrac12 (s'-s)^2 \kappa_1(s) + o((s'-s)^2). \]
Similarly, for every $s''$ close to $\varphi(s)$
\[ \gamma_2 (s'') = \gamma_1 (s) + (s''- \varphi(s)) \gamma'_2 (\varphi(s)) + \tfrac12 (s''-\varphi(s))^2 \kappa_2(\varphi(s)) + o((s''- \varphi(s))^2). \]
Since $s \notin N^*_1$, there is a sequence of parameters $s'_k \to s$, $s'_k \neq s$, $s_k \in X_1$ such that with $s''_k = \varphi(s''_k)$ we have $\gamma_1(s'_k) = \gamma_2(s''_k)$.
By the fact that $\varphi$ is continuously differentiable,
\[ s''_k - \varphi(s) = \varphi'(s) (s'_k - s) + o(s'_k - s). \]
So we obtain from the expansions for $\gamma_1, \gamma_2$ that
\begin{multline*}
 (s'_k - s) \gamma'_1(s) + o (s'_k - s) = \gamma_1 (s'_k) - \gamma_1 (s) \\
 = \gamma_2 (s''_k) - \gamma_2 (\varphi(s)) = \varphi'(s) (s'_k - s) \gamma'_2 (\varphi(s)) + o (s'_k - s).
\end{multline*}
This implies that $\gamma'_1(s) = \gamma'_2(\varphi(s)) \varphi'(s)$ and $\varphi'(s) = \pm 1$ follows from the fact that $|\gamma'_1(s)| = |\gamma'_2(s)| = 1$.

Next, consider the metric $\langle \cdot, \cdot \rangle_{\gamma_1(s)}$ at the point $\gamma_1 (s)$.
Use this metric to pair the expansions for $\gamma_1, \gamma_2$ with an arbitrary vector $v^* \in \IR^n$ that is orthogonal to $\gamma'_1(s)$ and hence also to $\gamma'_2(\varphi(s))$ (with respect to $\langle \cdot, \cdot \rangle_{\gamma_1(s)}$).
Then
\begin{multline*}
\tfrac12 (s'_k - s)^2 \big\langle \kappa_1(s), v^* \big\rangle_{\gamma_1(s)} + o ((s'_k -s)^2) = \big\langle \gamma_1(s'_k) - \gamma_1(s), v^* \big\rangle_{\gamma_1(s)} \\
= \big\langle \gamma_2 (s''_k) - \gamma_1 (s), v^* \big\rangle_{\gamma_1(s)} 
= \tfrac12 (s'_k - s)^2 \big\langle \kappa_2(\varphi(s)), v^* \big\rangle_{\gamma_1(s)} + o((s'_k - s)^2).
\end{multline*}
So $\langle \kappa_1(s), v^* \rangle_{\gamma_1(s)} = \langle \kappa_2(\varphi(s)), v^* \rangle_{\gamma_1(s)}$.
Since $\kappa_1(s), \kappa_2(\varphi(s))$ are orthogonal to $\gamma'_1(s)$ with respect to $\langle \cdot, \cdot \rangle_{\gamma_1(s)}$, we conclude that $\kappa_1(s) = \kappa_2(\varphi(s))$.
\end{proof}

In the remainder of this subsection, we state the solution of the Plateau problem for loops with (possibly infinitely many) self-intersections.
We will hereby always make use of the following terminology.

\begin{Definition}
Let $\gamma : S^1 \to M$ be a continuous and contractible loop.
A continuous map $f : D^2 \to M$ is called a \emph{solution to the Plateau problem for $\gamma$} if $f$ is smooth, harmonic and almost conformal on the interior of $D^2$ and if $\area f = A(\gamma)$ and if there is an orientation preserving homeomorphism $\varphi : S^1 \to S^1$ such that $f |_{S^1} = \gamma \circ \varphi$.
\end{Definition}

We will also need a variation of the Douglas-type condition.

\begin{Definition}[Douglas-type condition] \label{Def:DouglasCondition}
Let $\gamma : S^1 \to M$ be a piecewise $C^1$ immersion that is contractible in $M$.
We say that $\gamma$ \emph{satisfies the Douglas-type condition} if for any distinct pair of parameters $s, t \in S^1$, $s \neq t$ with $\gamma(s) = \gamma(t)$ the following is true:
Consider the loops $\gamma_1, \gamma_2$ that arise from restricting $\gamma$ to the arcs of $S^1$ between $s$ and $t$.
Then
\[ A(\gamma) < A(\gamma_1) + A(\gamma_2). \]
\end{Definition}

We can now state a slightly more general solution of the Plateau problem.

\begin{Proposition} \label{Prop:PlateauProblem}
Consider a loop $\gamma : S^1 \to M$ that is a piecewise $C^1$-immersion and that is contractible in $M$.
Assume first that $\gamma$ satisfies the Douglas-type condition.
Then the following holds.
\begin{enumerate}[label=(\alph*)]
\item There is a solution $f : D^2 \to M$ to the Plateau problem for $\gamma$.
\item If $\gamma$ has regularity $C^{1,1}$ on $U \cap S^1$ for some open subset $U \subset D^2$ then for every $\alpha < 1$ the map $f$ (from assertion (a)) locally has regularity $C^{1, \alpha}$ on $U$.
Moreover, the restriction $f |_{S^1}$ has non-vanishing derivative on $U \cap S^1$ away from finitely many branch points.

Similarly, if $\gamma$ has regularity $C^{m, \alpha}$ for some $m \geq 2$ and $\alpha \in (0,1)$ on $U \cap S^1$, then $f$ locally has regularity $C^{m, \alpha}$ on $U$.
\item Assume that $\gamma_k : S^1 \to M$ is a sequence of continuous maps that uniformly converge to $\gamma$.
Moreover, assume that each $\gamma_k$ is $C$-Lipschitz for some uniform $C < \infty$.
Consider solutions to the Plateau problem $f_k : D^2 \to M$ for each such $\gamma_k$.
Then there are conformal maps $\psi_k : D^2 \to D^2$ such that the maps $f_k \circ \psi_k : D^2 \to M$ subconverge uniformly on $D^2$ and smoothly on $\Int D^2$ to a map $f : D^2 \to M$ that solves the Plateau problem for $\gamma$.

Furthermore, if $\gamma$ has regularity $C^{1,1}$ on $U \cap S^1$ for some open subset $U \subset D^2$ and $\gamma_k$ locally converges to $\gamma$ on $U \cap S^1$ in the $C^{1,\alpha}$ sense for some $\alpha \in (0,1)$, then the sequence $f_k$ actually converges to $f$ on $U$ in the $C^{1, \alpha'}$ sense for every $\alpha' < \alpha$.
\end{enumerate}
Next assume that $\gamma$ does not necessarily satisfy the Douglas-type condition and let $p$ be the number of places where $\gamma$ is not differentiable (i.e. where the right and left-derivatives don't agree).
Then there are finitely or countably infinitely many loops $\gamma_1, \gamma_2, \ldots : S^1 \to M$ that are piecewise $C^1$-immersions and contractible in $M$ such that:
\begin{enumerate}[label=(\alph*), start=4]
\item The loops $\gamma_i$ satisfy the Douglas-type condition.
\item Each $\gamma_i$ is composed of finitely many subsegments of $\gamma$ in such a way that each such subsegment of $\gamma$ is used at most once for the entire sequence $\gamma_1, \gamma_2, \ldots$.
\item For each $i$ let $p_i$ be the number of places where $\gamma_i$ is not differentiable.
Then $p_i = 2$ for all but finitely many $i$ and
\[ \sum_i (p_i - 2) \leq  p - 2. \]
\item We have
\[ A(\gamma) = \sum_i A(\gamma_i). \]
\item For any set of solutions $f_1, f_2, \ldots : D^2 \to M$ to the Plateau problems for $\gamma_1, \gamma_2, \ldots$ and every $\delta > 0$ there is a map $f_\delta : D^2 \to M$ and an open subset $D_\delta \subset D^2$ such that the following holds:
$f_\delta |_{S^1} = \gamma$ and $f_\delta$ restricted to each connected component of $D_\delta$ is a diffeomorphic reparameterization of some $f_i$ restricted to an open subset of $D^2$ in such a way that every $i$ is used for at most one component of $D_\delta$.
Moreover
\[ \area f_\delta |_{D^2 \setminus D_\delta} < \delta \qquad \text{and} \qquad \area f_\delta < A(\gamma) + \delta. \]
\end{enumerate}
\end{Proposition}

\begin{proof}
We first prove the first part of assertion (c).
Since $\gamma_k$ uniformly converges to $\gamma$ and the curves are uniformly Lipschitz, we can find maps $H_k : S^1 \times [0,1] \to M$ that are $C'$-Lipschitz for some uniform $C' < \infty$, smooth on $S^1 \times (0,1)$ and that satisfy $H_k (\cdot, 0) = \gamma$, $H_k (\cdot, 1) = \gamma_k$ and $\lim_{k \to \infty} \area H_k = 0$ (compare with the proof of Lemma \ref{Lem:existenceon1skeleton}).
So
\[ \lim_{k \to \infty} \area f_k = \lim_{k \to \infty} A(\gamma_k) = A (\gamma). \]
Next, recall that there are orientation preserving homeomorphisms $\varphi_k : S^1 \to S^1$ such that $f_k |_{S^1} = \gamma_k \circ \varphi_k$.
Let $s_1, s_2, s_3 \in S^1$ be three pairwise distinct points and choose orientation preserving conformal maps $\psi_k : D^2 \to D^2$ such that $\psi_k (s_i) = \varphi_k^{-1} (s_i)$ for all $i = 1, 2, 3$ and $k = 1, 2, \ldots$.
Then each map $f_k \circ \psi_k$ is still a solution to the Plateau problem for $\gamma_k$ and $(f_k \circ \psi_k) |_{S^1} = \gamma_k \circ \varphi_k \circ \psi_k$.
So we may replace $f_k$ by $f_k \circ \psi_k$ and $\varphi_k$ by $\varphi_k \circ \psi_k$ and assume in the following, without loss of generality, that $\varphi_k (s_i) = s_i$ for each $i = 1,2,3$ and $k = 1,2, \ldots$.

By compactness and since the maps $\varphi_k$ are monotone (i.e. $\varphi_k$ restricted to the arcs between $s_1, s_2, s_3$ is monotone), we may pass to a subsequence and assume that the $\varphi_k$ converge pointwise to some monotone map $\varphi : S^1 \to  S^1$ with $\varphi (s_i) = s_i$.
We claim that $\varphi$ is continuous.
Assume not.
Then there is a point $s_0 \in S^1$ such that the left and right limits $t_- = \lim_{s \nearrow s_0} \varphi(s)$, $t_+ = \lim_{s \searrow s_0} \varphi(s)$ at $s_0$ don't agree, i.e. $t_- \neq t_+$.
If $\gamma(t_-) \neq \gamma(t_+)$, then we can derive a contradiction as in \cite[Lemma 9.3.2]{Morrey-book}.
Note that due to almost conformality of $f_k$, its energy satisfies
\[ \int_{\Int D^2} | d f_k  |^2  = 2 \area f_k = 2A(\gamma_k). \]
It remains to consider the case $\gamma(t_-) = \gamma(t_+)$.
An inspection of the arguments of \cite[Lemma 9.3.2]{Morrey-book} shows that we can still derive a contradiction under the following assumption:
There are constants $d, \delta > 0$ such that for any $0 < \eps < \delta$ and sufficiently large $k$ (depending on $\eps$), any embedded smooth curve $\sigma : [0,1] \to D^2$ that connects a point in $[s_0 - \delta, s_0 - \eps]$ with a point in $[s_0 + \eps, s_0 + \delta]$ (in $S^1$) satisfies $\ell ( f_k  \circ \sigma ) \geq d$.

We will now assume that this assumption does not hold.
That is, for any $d, \delta > 0$ there is an $0 < \eps < \delta$ and a sequence $\sigma_k : [0,1] \to D^2$ of embedded smooth curves that connect a point in $[s_0 - \delta, s_0 - \eps]$ with a point in $[s_0 + \eps, s_0 + \delta]$ such that $\ell ( f_k  \circ \sigma_k ) < d$ for infinitely many $k$.
Note that since $\varphi_k \to \varphi$ pointwise and $\varphi$ is monotone, we can find for any $\eta > 0$ a $\delta > 0$ such that for any $0 < \eps < \delta$ and sufficiently large $k$ (depending on $\eps$) we have $| t_- - \varphi_k(s) | < \eta$ for all $s \in [s_0 - \delta, s_0 - \eps]$ and $| t_+ - \varphi_k(s) | < \eta$ for all $s \in [s_0 + \eps, s_0 + \delta]$.
Combining these two facts, we can pass to a subsequence and find a sequence of embedded smooth curves $\sigma_k : [0,1] \to D^2$ whose endpoints lie in $S^1$ such that $\sigma_k(0), \sigma_k(1) \to s_0$, $\varphi_k(\sigma_k(0)) \to t_-$, $\varphi_k (\sigma_k (1)) \to t_+$ and 
\[ \lim_{k \to \infty} \ell ( f_k \circ \sigma_k)  = 0. \]
We will now argue that such a scenario contradicts the Douglas-type condition for $\gamma$.
Let $\ov\gamma_1, \ov\gamma_2 : S^1 \to M$ be the loops arising from restricting $\gamma$ to the arcs $a_{1}, a_{2} \subset S^1$ between $t_-$ and $t_+$.
For every $k$ let $D_{1, k}, D_{2,k}$ be the closures of the two components of $D^2 \setminus  \sigma_k ([0,1])$ such that for each $i = 1,2$, the arc $\varphi_k (\partial D_{i,k} \cap \partial D^2)$ contains more and more points of $a_i$ as $k \to \infty$.
For each $i = 1,2$ and $k = 1, 2, \ldots$ we can combine $f_k |_{D_{i,k}}$ with $H_k$ restricted to the subset $(\partial D_{i,k} \cap \partial D^2) \times [0,1] \subset \partial D^2 \times [0,1]$, mollify around the seam and obtain a continuous map $f'_{i,k} : D^2 \to M$ whose restriction to the interior is smooth and bounded in $W^{1,2}$ such that
\begin{equation} \label{eq:fprimefD}
 \lim_{k \to \infty} \big( \area f'_{i,k} - f_k |_{D_{i,k}} \big)  = 0. 
\end{equation}
Moreover, $f'_{i,k} |_{\partial D^2}$ describes the loop that is the concatenation of the curves $\gamma |_{\varphi_k (\partial D_{i,k} \cap \partial D^2)}$, $f_k \circ \sigma_k$ and two curves corresponding to $H_k$ restricted to the two radial lines of $(\partial D_{i,k} \cap \partial D^2) \times [0,1]$, whose length goes to $0$ as $k \to \infty$.
So $f'_{i,k} |_{\partial D^2}$ can be obtained from $\ov\gamma_i$ by attaching a loop of length $l_k \to 0$ along a subsegment and deleting the overlap.
Using the isoperimetric inequality and (\ref{eq:fprimefD}), it follows that for some uniform $C'' < \infty$
\[ A(\ov\gamma_i) \leq \liminf_{k \to \infty} \big( \area f'_{i,k}  + C l_k^2 \big) = \liminf_{k \to \infty} \area f_k |_{D_{i,k}}. \]
Letting $k \to \infty$, yields
\[ A(\gamma_1) + A(\gamma_2) \leq \liminf_{k \to \infty} \big( \area f_k |_{D_{1,k}} + \area f_k |_{D_{2,k}}  \big) = \lim_{k \to \infty} \area f_k = A(\gamma), \]
which contradicts the Douglas-type condition.

Summarizing our findings, we have shown that $\varphi : S^1 \to S^1$ is continuous.
Since $\varphi$ is monotone and $\varphi (s_i) = s_i$ for $i = 1,2,3$, we deduce that $\varphi$ is also surjective and has mapping degree $1$.
Moreover, by the monotonicity of the $\varphi_k$, we obtain that the convergence $\varphi_k \to \varphi$ is actually uniform.
So $f_k |_{S^1}$ converges uniformly to $\gamma \circ \varphi$.
The subconvergence of the $f_k$ to a harmonic and conformal $f : D^2 \to M$ with $f |_{S^1} = \gamma \circ \varphi$ now follows as in the proof of \cite[Theorem 9.4.3]{Morrey-book}.
Note that in this proof, the sequence ``$z_n$'' coming from \cite[Lemma 9.4.8]{Morrey-book} can be chosen to be the sequence $f_k$ and \cite[Theorem 9.4.2]{Morrey-book} is redundant, since the $f_k$ are already energy minimizing.
The fact that $\gamma$ may have self-intersections does not create any issues, since it was only used in the proof of \cite[Lemma 9.4.8]{Morrey-book}.
In order to finish the proof of the first part of assertion (c), it only remains to show that $\varphi$ is injective, i.e. that $\varphi$ cannot be constant on a non-empty, open arc $a \subset S^1$.
Assume that such an arc $a$ existed and choose $p \in M$ such that $\{ p \} = f ( a) = \gamma (\varphi (a))$.
Let $\gamma^* : (-1,1) \to M$ be any smooth, embedded curve with $\gamma (0) = p$ and choose an open $U \subset D^2$ such that $p \in U \cap \partial D^2 \subset a$.
Using \cite{HH}, we obtain that $f$ must be constant on $U$, which is a contradiction.

Next, we prove assertion (a) using the first part of assertion (c).
By perturbing $\gamma$, we can find a sequence of smooth \emph{embeddings} $\gamma_k : S^1 \to M$ that are uniformly Lipschitz and that uniformly converge to $\gamma$.
Using \cite[Theorem 9.4.3]{Morrey-book} (see also \cite{Mor}), there is a solution $f_k : D^2 \to M$ to the Plateau problem for each $\gamma_k$.
By the first part of assertion (c), we can pass to a limit and obtain a solution to the Plateau problem for $\gamma$.

The proof of assertion (b) in the case in which $\gamma$ is $C^2$ on $U \cap S^1$ can be found in \cite{HH}.
We remark that in the case, in which $\gamma$ is only $C^{1,1}$ on $U \cap S^1$ and $g$ is locally flat on $U$, assertion (b) is a consequence of \cite{Kin}.
For our purposes, however, it is enough to note that the methods of the proof of \cite{HH} carry over to the case in which $\gamma$ is only $C^{1,1}$ on $U \cap S^1$.
We briefly point out how this can be done:
The first step in \cite{HH} consists of the choice of a local coordinate system $(x_1, \ldots, x_n)$ in which $\gamma$ is locally mapped to the $x_n$-axis.
For the subsequent estimates, this coordinate system has to be of class $C^2$.
In the case in which $\gamma$ is only $C^{1,1}$ on $U \cap S^1$, we can choose a sequence of coordinate systems $(x^k_1, \ldots, x^k_n)$ that are uniformly bounded in the $C^2$ sense, and that converge to a coordinate system $(x^\infty_1, \ldots, x^\infty_n)$ of regularity $C^{1,1}$ in every $C^{1,\alpha}$ norm and in this coordinate system $\gamma$ is locally mapped to the $x_n$-axis.
The minimal surface equation in the coordinate system $(x^k_1, \ldots, x^k_n)$ implies an equation of the form $|\triangle y^k| \leq \beta |\nabla y^k|^2$ for $y^k = (x^k_1, \ldots, x^k_{n-1}) \circ f$ where $\beta$ can be chosen independently of $k$.
Moreover, $y^k$ restricted to $U \cap S^1$ converges to $0$ in every $C^{1,\alpha}$ norm as $k \to \infty$.
Let $U''' \Subset U'' \Subset U' \Subset U$ be arbitrary compactly contained open subsets.
A closer look at the proof of the ``Hilfssatz'' in \cite{Hei} yields that for every $r > 0$ we have the estimate $|y^k| < C r$ on $U' \cap (D^2(1- r) \setminus D^2(1- 2r))$ if $k$ is large depending on $r$.
Here $C$ is independent of $k$.
It follows then that $\Vert y^k \Vert_{C^1 (U'' \cap D^2 (1-r))} < C$ for every $r > 0$ and large $k$.
This implies $\Vert y^\infty \Vert_{C^1 (U'')} < C$ and hence $\Vert y^k \Vert_{C^1(U'')} < 2C$ for large $k$. 
Standard elliptic estimates applied to the equation $|\triangle y^k| < 4\beta C^2$ then yield that $\Vert y^k \Vert_{C^{1, \alpha}(U''')} < C'$ for large $k$.
The regularity of $x^k_n \circ f$ and the fact that branch points are isolated also follow similarly as in \cite{HH}.

The second  part of assertion (c) follows in a similar manner.
We just need to choose the local coordinate systems $(x^k_1, \ldots, x^k_n)$ such that both $(x^k_1, \ldots, x^k_n) \circ \gamma$ and $(x^k_1, \ldots, x^k_n) \circ \gamma_k$ locally converge to the $x_n$-axis in the $C^{1,\alpha}$ sense. 

Now consider the case in which $\gamma$ does not satisfy the Douglas-type condition.
Then the remaining assertions follow from the methods of Hass (\cite{Hass-Plateau}).
For completeness, we briefly recall his proof.

We will inductively construct a (finite or infinite) sequence of straight segments $\sigma_1, \sigma_2, \ldots \subset D^2$ between pairs of points $s,t \in S^1$ with $\gamma(s) = \gamma(t)$, such that any two distinct segments don't intersect in their interior and such that the following holds for all $k \geq 0$:
Consider the (unique) extension $\gamma_k : S^1 \cup \sigma_1 \cup \ldots \cup \sigma_k  \to M$ of the map $\gamma$ that is constant on each $\sigma_i$.
Then we assume that the sum $A (\gamma_k |_{\partial \Omega})$ over all connected components $\Omega \subset \Int D^2 \setminus (\sigma_1 \cup \ldots \cup \sigma_k)$ is equal to $A(\gamma)$.
(Note that every such component is bounded by some of the $\sigma_i$ and some arcs of $S^1$.)

Having constructed segments $\sigma_1, \ldots, \sigma_k$, we will choose $\sigma_{k+1}$ as follows:
Consider all components $\Omega \subset \Int D^2 \setminus (\sigma_1 \cup \ldots \cup \sigma_k)$ such that $\gamma_k |_{\partial \Omega}$ does not satisfy the Douglas-type condition (or to be precise, such that the loop that is composed of the restriction of $\gamma$ to $S^1 \cap \partial \Omega$ does not satisfy the Douglas-type condition).
If there is no such $\Omega$, then we are done.
Otherwise we pick an $\Omega$ for which $\ell (\gamma |_{S^1 \cap \partial\Omega})$ is maximal.
By our assumption, we can find a straight segment $\sigma \subset D^2$ connecting two distinct parameters $s,t \in S^1 \cap \partial \Omega$ such that if we denote by $\Omega', \Omega''$ the two components of $\Omega \setminus \sigma'$, then
\begin{equation} \label{eq:OmegaOmegas}
 A (\gamma_k |_{\partial \Omega}) = A (\gamma_k |_{\partial \Omega'}) + A (\gamma_k |_{\partial \Omega''}).
\end{equation}
So if we choose $\sigma_{k+1} = \sigma$ for any such $\sigma$, then the extension $\gamma_{k+1} : S^1 \cup \sigma_1 \cup \ldots \cup \sigma_{k+1} \to M$ still satisfies the same assumption as above.
Now pick $\sigma$ amongst all such straight segments such that $\min \{ \ell (\gamma |_{S^1 \cap \partial\Omega'}), \ell (\gamma |_{S^1 \cap \partial\Omega''}) \}$ is larger than $\frac12$ times the supremum of this quantity over all such $\sigma$ and set $\sigma_{k+1} = \sigma$.

Having constructed the sequence $\sigma_1, \sigma_2, \ldots$, we let $X \subset D^2$ be the closure of $\sigma_1 \cup \sigma_2 \cup \ldots$ and we let $\gamma_X : S^1 \cup X \to M$ be the direct limit of all extensions $\gamma_k$.
Then all components $\Omega \subset \Int D^2 \setminus X$ are bounded by finitely many straight segments and arcs of $S^1$.
We now show that $A(\gamma)$ is equal to the sum of $A(\gamma_X |_{\partial \Omega})$ over all such components:
Let $\Omega_1, \ldots, \Omega_N$ be arbitrary, pairwise distinct components of $\Int D^2 \setminus X$.
Then there is a $k_0$ such that for all $k > k_0$ these components lie in different components $\Omega_{1, k}, \ldots, \Omega_{N, k}$ of $\Int D^2 \setminus (\sigma_1 \cup \ldots \cup \sigma_k)$.
Moreover $\Omega_{j, k} \to \Omega_j$ as $k \to \infty$.
So $\lim_{k \to \infty} A (\gamma_X |_{\partial \Omega_{j, k}} ) = A (\gamma_X |_{\partial\Omega_j})$ for each $j = 1, \ldots, N$.
Since the choice of the $\Omega_j$ was arbitrary, this shows that the sum of $A (\gamma_X |_{\partial \Omega})$ over all connected components $\Omega \subset \Int D^2 \setminus X$ is not larger than $A(\gamma)$.
The other direction is follows from the the subadditivity of $A$ applied to a large but finite number of components of $\Int D^2 \setminus X$ along with an isoperimetric estimate bounding the area of the remaining components.

Next, we show that for each component $\Omega \subset \Int D^2 \setminus X$, the loop $\gamma_X |_{\partial \Omega}$ satisfies the Douglas-type condition.
If not, then we could separate $\Omega$ into two non-empty components $\Omega', \Omega''$ along a straight line $\sigma$ between two parameters $s,t \in S^1$ for which $\gamma(s) = \gamma(t)$ such that (\ref{eq:OmegaOmegas}) holds for $\gamma_X$ instead of $\gamma_k$.
Choose a sequence $\Omega_k \subset \Int D^2 \setminus (\sigma_1 \cup \ldots \cup \sigma_k)$ such that $\Omega_1 \supset \Omega_2 \supset \ldots$ and such that $\Omega_k \to \Omega$ as $k \to \infty$.
Let moreover $\Omega'_k, \Omega''_k$ be the components of $\Omega_k \setminus \sigma$ such that $\Omega'_k \to \Omega'$ and $\Omega''_k \to \Omega''$.
Then $\lim_{ k \to \infty } A (\gamma_k |_{\partial \Omega_k}) = A (\gamma_X |_{\partial \Omega})$ and $\lim_{ k \to \infty } A (\gamma_k |_{\partial \Omega'_k}) = A (\gamma_X |_{\partial \Omega'})$ and $\lim_{ k \to \infty } A (\gamma_k |_{\partial \Omega''_k}) = A (\gamma_X |_{\partial \Omega''})$.
Moreover, for all $k \geq 1$
\begin{multline*}
 A (\gamma_1 |_{\partial \Omega'_1}) + A (\gamma_1 |_{\partial \Omega''_1}) \leq A (\gamma_k |_{\partial \Omega'_k}) + A(\gamma_k |_{\partial (\Omega'_1 \setminus \Omega'_k)}) + A (\gamma_k |_{\partial \Omega''_k}) + A(\gamma_k |_{\partial (\Omega''_1 \setminus \Omega''_k)}) \\
 = A (\gamma_k |_{\partial \Omega'_k}) + A (\gamma_k |_{\partial \Omega''_k}) + A(\gamma_k |_{\partial (\Omega_1 \setminus \Omega_k)}) \\
 = A (\gamma_k |_{\partial \Omega'_k}) + A (\gamma_k |_{\partial \Omega''_k}) + A(\gamma_k |_{\partial \Omega_1 }) - A(\gamma_k |_{\partial \Omega_k}).
\end{multline*}
Letting $k \to \infty$ yields
\[ A (\gamma_1 |_{\partial \Omega'_1}) + A (\gamma_1 |_{\partial \Omega''_1}) \leq A (\gamma_1 |_{\partial \Omega_1}). \]
Since the opposite inequality is trivially true, we must have equality.
This, however, yields a contradiction, because by our construction of the sequence $\sigma_1, \sigma_2, \ldots$ we must have picked $\sigma$ earlier and hence $\sigma_k = \sigma$ for some $k$.

Assertions (d), (e) and (g) are direct consequences of the construction.
By the fact that $\gamma$ is a piecewise immersion, we can deduce that all but finitely many components of $\Omega \subset \Int D^2 \setminus X$ are bounded by exactly two straight segments and two arcs.
Assertion (f) follows by counting edges and vertices.
Finally, the functions $f_\delta$ from assertion (h) can be constructed by parameterizing the solutions $f_i$ by the corresponding component of $\Int D^2 \setminus X$ and mollifying.
\end{proof}

The following variational property is a direct consequence of assertion (h) and will be used twice in this paper.

\begin{Lemma} \label{Lem:variationofA}
Consider a contractible, piecewise $C^1$-immersion $\gamma : S^1 \to M$, let $\gamma_i$ be the loops from the second part of Proposition \ref{Prop:PlateauProblem} and consider solutions $f_i : D^2 \to M$ to the Plateau problem for each $\gamma_i$.
Let $(g_t)_{t \in [0,\varepsilon)}$ be a smooth family of Riemannian metrics such that $g_0 = g$ (not necessarily a Ricci flow) and denote by $A_t (\gamma)$ the infimum over the areas of all spanning disks with respect to the metric $g_t$.
Then in the barrier sense
\[ \frac{d}{dt^+} \Big|_{t = 0} A_t ( \gamma ) \leq \sum_i \int_{D^2} \frac{d}{dt} \Big|_{t = 0} d{\vol}_{f^*_i (g_t)} \]
Here $d{\vol}_{f^*_i (g_t)}$ denotes the volume form of the pull-back metric $f^*_i (g_t)$.
\end{Lemma}

\begin{proof}
Due to the smoothness of the family $(g_t)$, we can find a constant $C < \infty$ such that for any two vectors $v, w \in T M$ based at the same point and every $t \in [0, \varepsilon/2)$ we have
\[ \big| g_t (v,w) - g_0 (v,w) - t \partial_t  g_0 (v,w) \big| \leq C t^2 |v|_0 |w|_0. \]
Let now $\delta > 0$ be a small constant and consider the map $f_\delta : D^2 \to M$ from Proposition \ref{Prop:PlateauProblem}(h).
It follows that there is a constant $C' < \infty$, which is independent of $\delta$, such that for small $t$
\[ \bigg| \area_t f_\delta - \area_0 f_\delta - t \int_{D^2} \frac{d}{dt} \Big|_{t=0} d {\vol}_{f_\delta^* (g_t)} \bigg| \leq C' t^2 \area_0 f_\delta. \]
So we find that
\[ A_t (  \gamma ) \leq \area_0 f_\delta + t \int_{D^2} \frac{d}{dt} \Big|_{t=0} d {\vol}_{f_\delta^* (g_t)}  + C' t^2 \area_0 f_\delta . \]
By the properties of $f_\delta$ and the fact that the integrand in the previous integral is bounded by a multiple of $d{\vol}_{f^*_\delta (g_t)}$ independently of $\delta$, it follows that for fixed $t$ and for $\delta \to 0$ the right hand side of the previous inequality goes to
\[ A_0 ( \gamma ) + t \sum_i \int_{D^2} \frac{d}{dt} \Big|_{t = 0} d{\vol}_{f^*_i (g_t)} + C' t^2 A_0 ( \gamma ). \]
This yields the desired barrier.
\end{proof}

\subsection{The structure of a minimizer along the 1-skeleton} \label{subsec:1skeletonstruc}
Consider now again the $C^{1,1}$ regular map $f : V^{(1)} \to M$ from subsection \ref{subsec:regularityon1skeleton}.
The goal of this subsection is to derive a variational identity in the spirit of (\ref{eq:simplenuiskappa}).
However, due to possible self-intersections of $f$, this undertaking becomes a quite delicate issue and it will be important to analyze the combinatorics of these self-intersections.
Note that, at least a priori, there could be infinitely many such self-intersections and the set of self-intersections can have positive measure (and possibly empty interior).
Our main result will be Lemma \ref{Lem:almosteverywherealong1skeleton} and inequality (\ref{eq:nutimeskappageq0}) therein, which will be needed subsequently.
At this point we recall that by definition $f |_{\partial V} = f_0 |_{\partial V}$ is a smooth embedding.
So no edge at the boundary has a self-intersection and any two distinct edges may only intersect in their endpoints.

We denote by $F_1, \ldots, F_n$ the faces and by $E_1, \ldots, E_m$ the edges of $V$ in such a way that $E_1, \ldots, E_{m_0}$ are the edges of $\partial V$.
For every $k = 1, \ldots, m$ let $l_k$ be the length of $f |_{E_k}$ and let $\gamma_k : [0, l_k] \to M$ be a parameterization of $f |_{E_k}$ by arclength.
Since the maps $\gamma_k$ have regularity $C^{1,1}$ (see Lemma \ref{Lem:regularityon1skeleton}), we can find for each $k = 1, \ldots, n$ a vector field $\kappa_k : [0, l_k] \to TM$ along $\gamma_k$ (i.e. $\kappa_k(s) \in T_{\gamma_k(s)}M$) that equals the geodesic curvature of $\gamma_k$ almost everywhere (see Lemma \ref{Lem:speedandcurvatureagreeonintersection}).

Next, we apply Proposition \ref{Prop:PlateauProblem} for each loop $f |_{\partial F_j}$ ($j=1, \ldots, n$) and obtain loops $\gamma_{j,1}, \gamma_{j,2}, \ldots$, which satisfy assertions (d)--(h) of this Proposition.
Without loss of generality, we may assume that each $\gamma_{j,i}$ is parameterized by arclength, i.e. that $\gamma_{j,i} : S^1(l_{j,i}) \to M$ where $l_{j,i}$ is the length of $\gamma_{j,i}$.
As before, we choose vector fields $\kappa_{j,i} : S^1(l_{j,i}) \to TM$ along each $\gamma_{j,i}$ that represent the geodesic curvature almost everywhere.
Now, let $f_{j,i} : D^2 \to M$ be an arbitrary solution to the Plateau problem for each loop $\gamma_{j,i}$.
Proposition \ref{Prop:PlateauProblem}(b) yields that $f_{j,i}$ is $C^{1,\alpha}$ up to the boundary except at the finitely many points where $\gamma_{j,i}$ is not differentiable.
So we can choose unit vector fields $\nu_{j,i} : S^1(l_{j,i}) \to TM$ along each $\gamma_{j,i}$ that are orthogonal to $\gamma_{j,i}$ and outward pointing tangential to $f_{j,i}$ everywhere except at finitely many points.

For each edge $E_k$ and each adjacent face $F_j$ we can consider the collection of subsegments of the $\gamma_{j,i}$ that lie on $E_k$.
These subsegments are pairwise disjoint and are equipped with the vector fields $\nu_{j,i}$.
We can hence construct a vector field along $\gamma_k$ that is equal to each of the $\nu_{j,i}$ on the corresponding subsegment and zero everywhere else.
Doing this for all faces $F_j$ that are adjacent to $E_k$ yields vector fields $\nu^{(1)}_k, \ldots, \nu^{(v_k)}_k : [0, l_k] \to TM$ along $\gamma_k$ where $v_k$ is the valency of $E_k$.
Note that $|\nu^{(u)}_k| \leq 1$ for all $k =1, \ldots, m$ and $u = 1, \ldots, v_k$.

With this notation at hand we can derive the following variation formula.

\begin{Lemma} \label{Lem:variationofV}
For every continuous vector field $X \in C^0(M; TM)$ that vanishes on $f (\partial V \cap V^{(0)})$ we have
\begin{multline*}
\Bigg| \sum_{k = 1}^m  \int_0^{l_k} \Big\langle \sum_{u = 1}^{v_k} \nu_k^{(u)} (s) , X_{\gamma_k(s)} \Big\rangle ds + \sum_{k=m_0+1}^m 
\bigg( - \int_0^{l_k} \big\langle \kappa_k (s), X_{\gamma_k(s)} \big\rangle ds \\
 -\big\langle \gamma'_k (0), X_{\gamma_k (0)} \big\rangle + \big\langle \gamma'_k (l_k), X_{\gamma_k (l_k)} \big\rangle \bigg) \Bigg| 
 \leq \sum_{k=1}^{m_0} \int_0^{l_k} \big| X_{\gamma_k (s)} \big| ds.
\end{multline*}
\end{Lemma}

\begin{proof}
Let first $X \in C^\infty (M; TM)$ be a smooth vector field that vanishes on $f ( \partial V \cap V^{(0)})$ and consider the smooth flow $\Phi : \IR \times M  \to M$, $\partial_t \Phi_t = X \circ \Phi_t$ of $X$.
Observe that $\Phi_t (x) = x$ for all $x \in f (\partial V \cap V^{(0)})$ and $t \in \IR$.
For each $t \in \IR$ let $f'_t : V^{(1)} \to M$ be the map that is equal to $\Phi_t \circ f |_{V^{(1)} \setminus \partial V}$ on $V^{(1)} \setminus \partial V$ and equal to $f |_{\partial V}$ on $\partial V$.
By Lemma \ref{Lem:existenceon1skeleton} for all $t \in \IR$
\[ A (f'_t |_{\partial F_1} ) + \ldots + A (f'_t |_{\partial F_n} ) + \ell (f'_t) \geq A^{(1)}(f_0) \]
where equality holds for $t = 0$.
So we obtain that in the barrier sense
\begin{equation} \label{eq:variationisnonnegative}
 \frac{d}{dt^+} \Big|_{t = 0} \big( A (f'_t |_{\partial F_1} ) + \ldots + A (f'_t |_{\partial F_n} ) + \ell (f'_t) \big) \geq 0.
\end{equation}

Next we compute the derivative of each term on the left hand side.
First note that for all $k = m_0 + 1, \ldots, m$
\begin{equation} \label{eq:variationoflength}
 \frac{d}{dt^+} \Big|_{t = 0} \ell(\Phi_t \circ \gamma_k ) = - \int_0^{l_k} \big\langle \kappa_k (s), X_{\gamma_k(s)} \big\rangle ds - \big\langle \gamma'_k (0), X_{\gamma_k(0)} \big\rangle + \big\langle \gamma'_k (l_k), X_{\gamma_k(l_k)} \big\rangle.
\end{equation}
Next, we estimate the derivatives of the area terms.
To do this, note that for each sufficiently differentiable map $h : D^2 \to M$ the area of $\Phi_t \circ h$ with respect to the metric $g$ is equal to the area of $h$ with respect to the metric $\Phi^*_t (g)$.
So we can use Lemma \ref{Lem:variationofA} and the first variation formula for the area to deduce that for each $j =1, \ldots, n$
\begin{equation} \label{eq:dtAffirstvariation}
  \frac{d}{dt^+} \Big|_{t = 0}  A ( \Phi_t \circ f |_{\partial F_j} ) \leq \sum_i \int_{D^2} \frac{d}{dt} \Big|_{t=0} d\vol_{f^*_{j,i} (\Phi_t^* (g))} = \sum_i \int_{D^2} \DIV_{f_{j,i}} (X \circ f_{j,i}).
\end{equation}
Here
\[ \DIV_{f_{j,i}} (X \circ f_{j,i}) = \sum_{u=1}^2 \big\langle \nabla_{df_{j,i} (e_u)} (X \circ f_{j,i} ), df_{j,i} (e_u) \big\rangle \]
for an orthonormal frame field $e_1, e_2$ on $D^2$ (note that the due to almost conformality, the volume form $d \vol_{f^*_{j,i} (g)}$ cancels with the inverse of $f_{j,i}^* (g)$).
Since $f_{j,i}$ is harmonic, we have
\[ \DIV_{f_{j,i}} (X \circ f_{j,i}) = \sum_{u=1}^2 \nabla_{df_{j,i} (e_u)} \big\langle  (X \circ f_{j,i} ), df_{j,i} (e_u) \big\rangle = \sum_{u=1}^2 \nabla_{e_u} \big\langle X \circ f_{j,i}, df_{j,i} (e_u) \big\rangle. \]
So by Stokes' Theorem
\[ \int_{D^2} \DIV_{f_{j,i}} (X \circ f_{j,i}) = \int_{\partial D^2} \big\langle X \circ f_{j,i}, df_{j,i} (s) \big\rangle ds = \int_{S^1(l_{j,i})} \big\langle \nu_{j,i} (s), X_{\gamma_{j,i} (s)} \big\rangle ds, \]
where in the second term $s \in \partial D^2$ is viewed both as a point in $\partial D^2$ and a unit tangent vector.
Plugging this back into (\ref{eq:dtAffirstvariation}), yields
\begin{equation} \label{eq:dtAplussum}
 \frac{d}{dt^+} \Big|_{t = 0}  A ( \Phi_t \circ f |_{\partial F_j} ) \leq   \sum_i \int_{S^1(l_{j,i})} \big\langle \nu_{j,i} (s), X_{\gamma_{j,i} (s)} \big\rangle ds. 
\end{equation}
Now consider for each $k = 1, \ldots, m_0$ the loop that is composed of $\gamma_k$ and $\Phi_t \circ \gamma_k$ (recall that the endpoints of $\gamma_k$ are left invariant by $\Phi_t$).
This loop bounds the disk that is described by the map $H_{t,k} : [0, l_k] \times [0, t] \to M$ with $(s,t') \mapsto \Phi_{t'} (\gamma_k(s))$.
Note that $\area H_{t,k} = t \int_0^{l_k} |X_{\gamma_k(s)} | ds + O(t^2)$ for small $t$.
Moreover, since each loop $f'_t |_{\partial F_j}$ can be obtained from $\Phi_t \circ f |_{\partial F_j}$ by possibly replacing some $\gamma_k$ by $\Phi_t \circ \gamma_k$, we have
\begin{multline*}
 A(f'_t |_{\partial F_1} ) + \ldots + A(f'_t |_{\partial F_n} ) \\ \leq A( \Phi_t \circ f |_{\partial F_1} ) + \ldots + A( \Phi_t \circ f |_{\partial F_n} ) + \area H_{t,1} + \ldots + \area H_{t, m_0}.
\end{multline*}
So taking the derivative at $t = 0$ yields together with (\ref{eq:dtAplussum})
\begin{multline*}
 \frac{d}{dt^+} \Big|_{t = 0} \big( A (f'_t |_{\partial F_1} ) + \ldots + A (f'_t |_{\partial F_n} ) \big) \\ \leq \sum_{j=1}^m \sum_i \int_{S^1(l_{j,i})} \big\langle \nu_{j,i} (s), X_{\gamma_{j,i} (s)} \big\rangle ds  + \sum_{k=1}^{m_0}  \int_0^{l_k} \big| X_{\gamma_k(s)} \big| ds.
\end{multline*}
Together with (\ref{eq:variationisnonnegative}) and (\ref{eq:variationoflength}) this yields
\begin{multline*}
 \sum_{j=1}^m \sum_i \int_{S^1(l_{j,i})} \big\langle \nu_{j,i} (s), X_{\gamma_{j,i} (s)} \big\rangle ds + \sum_{k = m_0 + 1}^m \bigg( - \int_0^{l_k} \big\langle \kappa_k (s), X_{\gamma_k(s)} \big\rangle ds \\
 -\big\langle \gamma'_k (0), X_{\gamma_k (0)} \big\rangle + \big\langle \gamma'_k (l_k), X_{\gamma_k (l_k)} \big\rangle \bigg) + \sum_{k=1}^{m_0}  \int_0^{l_k} \big| X_{\gamma_k(s)} \big| ds \geq 0.
\end{multline*}
Note that by rearrangement
\[  \sum_{j=1}^m \sum_i \int_{S^1(l_{j,i})} \big\langle \nu_{j,i} (s), X_{\gamma_{j,i} (s)} \big\rangle ds = \sum_{k = 1}^m  \int_0^{l_k} \Big\langle \sum_{u = 1}^{v_k} \nu_k^{(u)} (s) , X_{\gamma_k(s)} \Big\rangle ds. \]
So our conclusions applied for $X$ and $-X$ show that the desired inequality holds for all smooth vector fields that vanish on $f ( \partial V \cap V^{(0)} )$.
By continuity it must also hold for all \emph{continuous} vector fields that vanish on $f ( \partial V \cap V^{(0)} )$.
\end{proof}

We can now use this inequality to derive the following identities.

\begin{Lemma} \label{Lem:almosteverywherealong1skeleton}
For every $x \in f(V^{(0)}) \setminus f ( V^{(0)} \cap \partial V)$ the (normalized) directional derivatives of $f$ at every vertex of $V^{(0)}$ that is mapped to $x$, in the direction of each adjacent edge, add up to zero.

Moreover, for every $k = 1, \ldots, m$ and for almost all $s \in [0, l_k]$ the following holds:
If $\gamma_k(s) \notin f(\partial V)$, then
\begin{equation} \label{eq:pointwisenukappa}
 \sum_{k' = 1}^m \sum_{\substack{s' \in E_{k'}  \\ f (s') = f(s)}}  \sum_{u = 1}^{v_{k'}} \nu_{k'}^{(u)} (s') - |f^{-1} (f(s)) | \cdot \kappa_{k} (s) = 0.
\end{equation}
Otherwise
\begin{equation} \label{eq:pointwisenukappaOnBoundary}
 \bigg| \sum_{k' = 1}^m \sum_{\substack{s' \in E_{k'}  \\ f (s') = f(s)}}  \sum_{u = 1}^{v_{k'}} \nu_{k'}^{(u)} (s') - \big( |f^{-1} (f(s)) | -1 \big) \cdot \kappa_{k} (s) \bigg| \leq 1.
\end{equation}
Furthermore, we have the integral inequality
\begin{equation} \label{eq:nutimeskappageq0}
 \sum_{j=1}^n \sum_i \int_{S^1 (l_{j, i})} \big\langle \nu_{j, i}(s) , \kappa_{j,i}(s) \big\rangle ds \geq - \sum_{k = 1}^{m_0} \int_0^{l_k} \big| \kappa_k (s) \big| ds.
\end{equation}
\end{Lemma}

\begin{proof}
Recall that all $\kappa_k$ and $\nu_{j, i}$ are uniformly bounded.
Let $X$ be a (not necessarily) continuous vector field on $M$ that vanishes on $f(\partial V \cap V^{(0)})$.
For any $\eps > 0$ let $X^{(\eps)}$ be a vector field that agrees with $X$ on $f(V^{(0)})$, that vanishes outside an $\eps$-neighborhood of $f( V^{(0)} )$ and that satisfies $|X^{(\eps)}| \leq C$ everywhere for some uniform constant $C < \infty$.
For example, $X^{(\eps)}$ can be obtained from $X$ by making $X$ continuous near each point of $f(V^{(0)})$ and multiplying with an appropriate cutoff function.
If we apply the variation formula in Lemma \ref{Lem:variationofV} to each such $X^{(\eps)}$, then the contribution of the integrals goes to zero as $\eps \to 0$, while the other two terms are independent of $\eps$.
So letting $\eps \to 0$ yields
\[ \sum_{k =1}^m \Big( -\big\langle \gamma'_k (0), X_{\gamma_k (0)} \big\rangle + \big\langle \gamma'_k (l_k), X_{\gamma_k (l_k)} \big\rangle \Big) = 0. \]
This implies the very first part of the claim and simplifies the variation formula:
For every continuous vector field $X \in C^0 (M; TM)$ we have
\begin{multline} \label{eq:variationwithoutvertices}
 \Bigg| \sum_{k = 1}^m  \int_0^{l_k} \Big\langle \sum_{u = 1}^{v_k} \nu_k^{(u)} (s) , X_{\gamma_k(s)} \Big\rangle ds - \sum_{k=m_0+1}^m 
 \int_0^{l_k} \big\langle \kappa_k (s), X_{\gamma_k(s)} \big\rangle ds 
 \Bigg| \\
 \leq \sum_{k=1}^{m_0} \int_0^{l_k} \big| X_{\gamma_k (s)} \big| ds.
\end{multline}

Choose $N < \infty$ large enough such that the following holds:
Each curve $\gamma_k$ restricted to a subinterval of length $\frac1N l_k$ is embedded and whenever two curves $\gamma_{k_1}, \gamma_{k_2}$ restricted to subintervals of length $\frac1{N} l_{k_1}, \frac1{N} l_{k_2}$ intersect, then we are in the situation of Lemma \ref{Lem:speedandcurvatureagreeonintersection}, i.e. we can find a coordinate chart $(U, (x_1, \ldots, x_n))$ that contains these subsegments and in which we can find a vector $v \in \IR^n$ with the property that $\langle \gamma'_{k_1}, v \rangle, \langle \gamma'_{k_2}, v \rangle \neq 0$ on both subsegments with respect to the euclidean metric.
Consider now the index set $I = \{ 1, \ldots, m \} \times \{ 0, \ldots, N-1 \}$ and define for every $(k, e) \in I$ and every subset $I' \subset I$ with $(k, e) \in I'$ the domain
\begin{multline*}
 \mathcal{D}_{k, e, I'} = \big\{ s \in [\tfrac{e}{N} l_k, \tfrac{e+1}{N} l_k] \;\; : \;\; \gamma_k (s) \in \gamma_{k'} \big( [\tfrac{e'}{N} l_{k'}, \tfrac{e'+1}{N} l_{k'}] \big) \\ \text{if and only if $(k',e') \in I'$} \big\}.
\end{multline*}
These sets are measurable and for all $(k, e) \in I$
\[ \mathop{\dot\bigcup}_{\substack{I' \subset I \\ (k, e) \in I'}} \mathcal{D}_{k, e, I'} = [\tfrac{e}{N} l_k, \tfrac{e+1}{N} l_k]. \]
Moreover, since $f |_{\partial V} = f_0 |_{\partial V}$ is injective, we find that $\mathcal{D}_{k,e,I'}$ is empty or finite whenever there are two distinct pairs $(k', e'), (k'', e'') \in I'$ for which $k', k'' \leq m_0$.

Consider now two pairs $(k_1, e_1), \linebreak[1] (k_2, e_2)$ and a subset $I' \subset I$ such that $(k_1, e_1), \linebreak[1] (k_2, e_2) \in I'$ and assume that $\mathcal{D}_{k_1, e_1, I'}$ (and hence also $\mathcal{D}_{k_2, e_2, I'}$) is non-empty.
We can now apply the second part of Lemma \ref{Lem:speedandcurvatureagreeonintersection} and obtain a continuously differentiable map $\varphi : [\frac{e_1}N, \frac{e_1+1}N] \to \IR$, whose derivative vanishes nowhere, for which the following holds:
$\varphi (\mathcal{D}_{k_1, e_1, I'}) = \mathcal{D}_{k_2, e_2, I'}$ and $\gamma_{k_1} (s) = \gamma_{k_2} (\varphi(s))$ for all $s \in \mathcal{D}_{k_1, e_1, I'}$.
Moreover, for almost every $s \in \mathcal{D}_{k_1, e_1, I'}$ we have $\varphi'(s) = \pm 1$ and $\kappa_{k_1} (s) = \kappa_{k_2} (\varphi(s))$.
So the following three identities hold for every continuous vector field $X \in C^0(M; TM)$
\begin{alignat}{1} 
 \int_{\mathcal{D}_{k_1, e_1, I'}} \big\langle \kappa_{k_1} (s), X_{\gamma_{k_1} (s)} \big\rangle ds &= \int_{\mathcal{D}_{k_2, e_2, I'}} \big\langle \kappa_{k_2} (s), X_{\gamma_{k_2} (s)} \big\rangle ds, \label{eq:kappaintoverdifferentDD} \displaybreak[2] \\
 \int_{\mathcal{D}_{k_1, e_1, I'}} \Big\langle \sum_{u = 1}^{v_{k_2}} \nu_{k_2}^{(u)} (\varphi(s)), X_{\gamma_{k_1}(s)} \Big\rangle ds &= \int_{\mathcal{D}_{k_2, e_2, I'}} \Big\langle \sum_{u = 1}^{v_{k_2}} \nu_{k_2}^{(u)} (s), X_{\gamma_{k_2}(s)}  \Big\rangle ds, \label{eq:nuintoverdifferentDD} \displaybreak[2] \\ 
 \int_{\mathcal{D}_{k_1, e_1, I'}} \Big\langle \sum_{u=1}^{v_{k_2}} \nu_{k_2}^{(u)} (\varphi(s)), \kappa_{k_1} (s) \Big\rangle ds &= \int_{\mathcal{D}_{k_2, e_2, I'}} \Big\langle \sum_{u=1}^{v_{k_2}} \nu_{k_2}^{(u)} (s), \kappa_{k_2} (s) \Big\rangle ds.  \label{eq:nukappaproductondifferentDD}
\end{alignat}

Next we express both sides of (\ref{eq:variationwithoutvertices}) as sums of integrals over the domains $\mathcal{D}_{k, e, I'}$.
\begin{multline*}
 \Bigg| \sum_{I' \subset I} \bigg( \sum_{(k,e) \in I'} \int_{\mathcal{D}_{k, e, I'}} \Big\langle \sum_{u = 1}^{v_k} \nu_k^{(u)} (s) , X_{\gamma_k(s)} \Big\rangle ds - \sum_{\substack{(k,e) \in I' \\ k > m_0}} \int_{\mathcal{D}_{k, e, I'}} \Big\langle \kappa_k (s), X_{\gamma_k(s)} \Big\rangle ds  \bigg) \Bigg| \\
 \leq \sum_{I' \subset I} \sum_{\substack{(k,e) \in I' \\ k \leq m_0}} \int_{\mathcal{D}_{k,e,I'}} \big| X_{\gamma_k(s)} \big| ds.
\end{multline*}
We will now group integrals whose values are the same.
To do this set $I_0 = \{ 1, \ldots, m_0 \} \times \{ 0, \ldots, N-1 \}$ and for each $\emptyset \neq I' \subset I$ choose a pair $(k_{I'}, e_{I'}) \in I'$ such that $(k_{I'}, e_{I'}) \in I_0$ whenever $I' \cap I_0 \neq \emptyset$.
Using (\ref{eq:kappaintoverdifferentDD}) and (\ref{eq:nuintoverdifferentDD}) we may then express the integrals over the domains $\mathcal{D}_{k, e, I'}$ in the last inequality in terms of integrals over the domains $\mathcal{D}_{k_{I'}, e_{I'}, I'}$.
This yields
\begin{multline} \label{eq:variationsplitupinDD}
\Bigg| \sum_{\emptyset \neq I' \subset I} \int_{\mathcal{D}_{k_{I'} ,e_{I'} , I'}} \Big\langle \sum_{k =1}^m  \sum_{\substack{s' \in E_k  \\ f (s') = f(s)}} \sum_{u = 1}^{v_k} \nu_k^{(u)} (s') - |I' \cap (I \setminus I_0)| \cdot \kappa_{k_{I'}} (s), X_{\gamma_{k_{I'}}(s)} \Big\rangle ds \Bigg|  \\
 \leq \sum_{\substack{\emptyset \neq I' \subset I \\ I' \cap I_0 \neq \emptyset}} \int_{\mathcal{D}_{k_{I'} ,e_{I'} , I'}} \big| X_{\gamma_k(s)} \big| ds.
\end{multline}
Note that all summands involving $\emptyset \neq I' \subset I$ for which $I' \cap I_0$ contains more than one element vanish since those consist of integrals over a finite set.
So for all remaining summands and all $(k,e) \in I' \subset I$ for almost every $s \in \mathcal{D}_{k,e, I'}$ the quantity $|I' \cap (I \setminus I_0)|$ is equal to $|f^{-1} (f (s))|$ if $\gamma_k (s) \notin f(\partial V)$ (or equivalently if $I' \cap I_0 = \emptyset$) or equal to $|f^{-1} (f (s))| - 1$ if $\gamma_k (s) \in f(\partial V)$ (or equivalently if $|I' \cap I_0| = 1$).
So the first factor in the scalar product on the left hand side of (\ref{eq:variationsplitupinDD}) is equal to the left hand side of equation (\ref{eq:pointwisenukappa}) or (\ref{eq:pointwisenukappaOnBoundary}), depending on $I'$, almost everywhere.

We will now show by induction on $|I'|$ that for every $\emptyset \neq I' \subset I$ equation (\ref{eq:pointwisenukappa}) or (\ref{eq:pointwisenukappaOnBoundary}) holds for almost every $s \in \mathcal{D}_{k_{I'}, n_{I'}, I'}$.
Using the previous conclusions, which related $\mathcal{D}_{k, e, I'}$ to $\mathcal{D}_{k_{I'}, e_{I'}, I'}$ for any other $(k,e) \in I'$, this will then imply the desired statement.
So let $\emptyset \neq I^* \subset I$ and assume that for all $\emptyset \neq I' \subsetneq I^*$ equation (\ref{eq:pointwisenukappa}) or (\ref{eq:pointwisenukappaOnBoundary}) holds for almost every $s \in \mathcal{D}_{k_{I'}, n_{I'}, I'}$.
This implies that the terms involving subsets $I'$ in the sums on both sides of the inequality (\ref{eq:variationsplitupinDD}) vanish whenever $\emptyset \neq I' \subsetneq I$ and $I' \cap I_0 = \emptyset$.

Consider now some $s_0 \in \mathcal{D}_{k_{I^*}, e_{I^*}, I^*}$.
Then we can find an open neighborhood $U \subset M$ around $\gamma_{k_{I^*} } (s_0)$ such that $\gamma_k([\frac{e}{N} l_k, \frac{e+1}{N} l_k]) \cap U \neq \emptyset$ if and only if $(k, e) \in I^*$.
So as long as $X \in C^0 ( M ; TM)$ is supported in $U$, the summands in (\ref{eq:variationsplitupinDD}) involving $\emptyset \neq I' \subset I$ with $\emptyset \neq I' \not\subset I^*$ vanish.
Therefore the only summands that are not a priori zero are the summand involving the subset $I' = I$ and all proper subsets $I' \subsetneq I^*$ for which $| I' \cap I_0 | = 1$.

Consider first the case in which $I^* \cap I_0 = \emptyset$.
Then the previous conclusion implies that only the summand involving $I^*$ on the left hand side of (\ref{eq:variationsplitupinDD}) is not a priori zero and that the right hand side of this equation is zero.
So
\[ \int_{\mathcal{D}_{k_{I^*} ,e_{I^*} , l^*}} \Big\langle \sum_{k =1}^m  \sum_{\substack{s' \in E_k  \\ f (s') = f(s)}} \sum_{u = 1}^{v_k} \nu_k^{(u)} (s') - |f^{-1} (f(s))| \cdot \kappa_{k_{I^*}} (s), X_{\gamma_{k_{I^*}}(s)} \Big\rangle ds = 0 \]
for all $X \in C^0(M; TM)$ that are supported in $U$.
Since $\gamma_{k_{I^*}(s)}$ restricted to $[\frac{e_{I^*}}N l_{k_{I^*}}, \frac{e_{I^*} + 1}N l_{k_{I^*}}]$ is an embedding, this implies that
\[ \int_{\mathcal{D}_{k_{I^*} ,e_{I^*} , l^*}} \Big\langle \sum_{k =1}^m  \sum_{\substack{s' \in E_k  \\ f (s') = f(s)}} \sum_{u = 1}^{v_k} \nu_k^{(u)} (s') - |f^{-1} (f(s))| \cdot \kappa_k (s), X (s) \Big\rangle ds = 0 \]
for every compactly supported continuous vector function $X \in C^0 ( \gamma_{k_{I^*}}^{-1} (U) \cap \linebreak[1] [\frac{e_{I^*}}N l_{k_{I^*}}, \linebreak[1] \frac{e_{I^*} + 1}N l_{k_{I^*}}])$.
So (\ref{eq:pointwisenukappa}) holds almost everywhere on $\mathcal{D}_{k_{I^*} ,e_{I^*} , l^*} \cap \gamma_{k_{I^*}}^{-1} (U) \cap \linebreak[1] [\frac{e_{I^*}}N l_{k_{I^*}}, \linebreak[1] \frac{e_{I^*} + 1}N l_{k_{I^*}}]$.
Since $s_0$ was chosen arbitrarily within $\mathcal{D}_{k_{I^*}, e_{I^*}, I^*}$, this shows that (\ref{eq:pointwisenukappa}) holds for almost every $s \in \mathcal{D}_{k_{I^*} ,e_{I^*} , l^*}$, which finishes the induction in the first case.

Next consider the case in which $I^* \cap I_0 = \{ (k_{I^*}, e_{I^*}) \}$.
Then for every non-zero summand in (\ref{eq:variationsplitupinDD}) involving $I'$ we have $(k_{I'}, e_{I'}) = (k_{I^*}, e_{I^*}) =: (k_0, e_0)$.
Since the union of all domains $\mathcal{D}_{k_0, e_0, I'}$ for which $(k_0, e_0) \in I'$ is equal to the interval $[\frac{e_0}N l_{k_0}, \frac{e_0 + 1}{N} l_{k_0}]$, inequality (\ref{eq:variationsplitupinDD}) implies that
\begin{multline*}
\Bigg| \int_{\frac{e_0}N l_{k_0}}^{\frac{e_0+1}N l_{k_0}}  \Big\langle \sum_{k =1}^m  \sum_{\substack{s' \in E_k  \\ f (s') = f(s)}} \sum_{u = 1}^{v_k} \nu_k^{(u)} (s') - \big( | f^{-1} (f(s)) \big| - 1\big) \cdot \kappa_{k_0} (s), X_{\gamma_{k_0}(s)} \Big\rangle ds \Bigg| \\
 \leq \int_{\frac{e_0}N l_{k_0}}^{\frac{e_0+1}N l_{k_0}} \big| X_{\gamma_k(s)} \big| ds
\end{multline*}
for all $X \in C^0(M; TM)$ that are supported in $U$.
As in the first case, we conclude that
\begin{multline*}
\Bigg| \int_{\frac{e_0}N l_{k_0}}^{\frac{e_0+1}N l_{k_0}}  \Big\langle \sum_{k =1}^m  \sum_{\substack{s' \in E_k  \\ f (s') = f(s)}} \sum_{u = 1}^{v_k} \nu_k^{(u)} (s') - \big( | f^{-1} (f(s)) \big| - 1\big) \cdot \kappa_{k_0} (s), X (s) \Big\rangle ds \Bigg| \\
 \leq \int_{\frac{e_0}N l_{k_0}}^{\frac{e_0+1}N l_{k_0}} \big| X(s) \big| ds
\end{multline*}
for every compactly supported continuous vector function $X \in C^0 ( \gamma_{k_0}^{-1} (U) \cap \linebreak[1] [\frac{e_0}N l_{k_0}, \linebreak[1] \frac{e_0 + 1}N l_{k_0}])$.
This implies that (\ref{eq:pointwisenukappaOnBoundary}) holds for almost all $s \in \mathcal{D}_{k_0, e_0, I^*} \subset [\frac{e_0}N l_{k_0}, \frac{e_0 + 1}{N} l_{k_0}]$ and finishes the induction in the second case.

Finally, we prove (\ref{eq:nutimeskappageq0}).
Observe that by rearrangement we have
\[ \sum_{j=1}^n \sum_i \int_{S^1(l_{j, i})} \langle \nu_{j,i} (s), \kappa_{j,i} (s) \rangle ds = \sum_{k=1}^m \sum_{u=1}^{v_k} \int_0^{l_k} \big\langle \nu_k^{(u)} (s), \kappa_k (s) \big\rangle ds. \]
Using (\ref{eq:nukappaproductondifferentDD}) we conclude further that
\begin{multline*}
\sum_{j=1}^n \sum_i \int_{S^1 (l_{j, i})} \big\langle \nu_{j, i}(s) , \kappa_{j,i}(s) \big\rangle ds = 
 \sum_{k=1}^m \sum_{u=1}^{v_k} \int_0^{l_k} \big\langle \nu_k^{(u)} (s), \kappa_k (s) \big\rangle ds \\ 
 = \sum_{I' \subset I} \sum_{(k,e) \in I'} \int_{\mathcal{D}_{k, e, I'}} \Big\langle \sum_{u = 1}^{v_k} \nu_k^{(u)} (s) , \kappa_k(s) \Big\rangle ds \\
 = \sum_{\emptyset \neq I' \subset I} \int_{\mathcal{D}_{k_{I'}, e_{I'}, I'}} \Big\langle \sum_{k = 1}^m \sum_{\substack{s' \in E_k  \\ f (s') = f(s)}} \sum_{u = 1}^{v_k} \nu_k^{(u)} (s'), \kappa_{k_{I'}} (s) \Big\rangle ds.
\end{multline*}
We now apply (\ref{eq:pointwisenukappa}) to all summands for which $I' \cap I_0 = \emptyset$ and (\ref{eq:pointwisenukappaOnBoundary}) to all summands for which $I' \cap I_0 \neq \emptyset$.
Then we obtain that the right hand side of the previous equation is bounded from below by
\begin{multline*}
 \sum_{\substack{\emptyset \neq I' \subset I \\ I' \cap I_0 = \emptyset}} \int_{\mathcal{D}_{k_{I'}, n_{I'}, I'}} |I'| \cdot \big\langle \kappa_{k_{I'}} (s), \kappa_{k_{I'}} (s) \big\rangle ds   \\
+ \sum_{\substack{\emptyset \neq I' \subset I \\ I' \cap I_0 \neq \emptyset}} \int_{\mathcal{D}_{k_{I'}, n_{I'}, I'}} \Big( (|I'|-1) \cdot 
\langle \kappa_{k_{I'}} (s), \kappa_{k_{I'}} (s) \big\rangle - \big| \kappa_{k_{I'}} (s) \big| \Big) ds \displaybreak[1] \\
\geq - \sum_{\substack{\emptyset \neq I' \subset I \\ I' \cap I_0 \neq \emptyset}} \int_{\mathcal{D}_{k_{I'}, n_{I'}, I'}} \big| \kappa_{k_{I'}} (s) \big| ds
= - \sum_{k = 1}^{m_0} \int_0^{l_k} \big| \kappa_k (s) \big| ds .
\end{multline*}
This establishes the claim.
\end{proof}

\subsection{Summary}
We conclude this section by summarizing the important results that are needed in section \ref{sec:areaevolutionunderRF}.

\begin{Proposition} \label{Prop:existenceofVminimizer}
Consider a compact Riemannian manifold $(M,g)$ with $\pi_2(M) \linebreak[0] = 0$.
Let $V$ be a finite simplicial complex whose faces are $F_1, \ldots, F_n$ and $f_0 : V \to M$ a continuous map such that $f_0 |_{\partial V}$ is a smooth embedding.
Furthermore, let $\gamma_k : [0, l_k] \to M$, $(k = 1, \ldots, m_0)$ be arclength parameterizations of $f$ restricted to the edges of $\partial V$ and $\kappa_k : [0, l_k] \to TM$ the geodesic curvature of $\gamma_k$

Then the following is true:
\begin{enumerate}[label=(\alph*)]
\item There is a map $f : V^{(1)} \to M$ that restricted to every edge $E \subset V^{(1)}$ is a $C^{1,1}$-immersion such that $f$ is homotopic to $f_0 |_{V^{(1)}}$ relative to $\partial V$ and
\[ A( f|_{\partial F_1} ) + \ldots + A( f|_{\partial F_n} ) + \ell (f)  = A^{(1)} (f_0). \]
\item Consider for each $j = 1, \ldots, n$ the loop $f |_{\partial F_j}$ and apply Proposition \ref{Prop:PlateauProblem} to obtain the loops $\gamma_{j,i} : S^1 (l_{j,i}) \to M$.
Let $p_{j,i}$ be the (finitely many) number of places where $\gamma_{j,i}$ is not differentiable.
Then $p_{j,i} = 2$ for almost every $i, j$ and
\[  \sum_i (p_{j,i} - 2) \leq  1. \]
\item For each loop $\gamma_{j,i}$, as defined in assertion (b), consider an arbitrary solution $f_{j,i} : D^2 \to M$ to the associated Plateau problem with respect to the metric $g$.
Consider moreover a smooth family $(g_t)_{t \in [0, \varepsilon)}$ of metrics with $g_0 = g$ and denote by $A_t ( \cdot )$, the infimal area $A( \cdot )$ with respect to the metric $g_t$.
Then for any $j = 1, \ldots, n$ we have in the barrier sense
\[ \frac{d}{dt^+} \Big|_{t = 0} A_t (f |_{\partial F_j} ) \leq \sum_i \int_{D^2} \frac{d}{dt} \Big|_{t = 0} d\vol_{f_{j,i}^* (g_t)}. \]
\item For each loop $\gamma_{j, i}$ the geodesic curvature $\kappa_{j,i} : S^1 (l_{j,i}) \to TM$ is defined almost everywhere.
Consider again the maps $f_{j,i} : D^2 \to M$ from before and let $\nu_{j,i} : S^1 (l_{j,i}) \to TM$ be unit vector fields along $\gamma_{j,i}$ that are orthogonal to $\gamma_{j,i}$ and outward pointing tangential to $f_{j,i}$.
Then
\[ \sum_{j=1}^n \sum_i \int_{S^1 (l_{j, i})} \big\langle \nu_{j, i}(s) , \kappa_{j,i}(s) \big\rangle ds \geq - \sum_{k = 1}^{m_0} \int_0^{l_k} \big| \kappa_k (s) \big| ds. \]
\end{enumerate}
\end{Proposition}

\begin{proof}
Assertion (a) is a consequence of Lemma \ref{Lem:existenceon1skeleton}, and the preceding discussion, and Lemma \ref{Lem:regularityon1skeleton}.
Assertion (b) is a restatement of Proposition \ref{Prop:PlateauProblem}(f).
For this note that $f |_{\partial F_j}$ is differentiable everywhere except possibly at its three corners.
Assertions (c) is a restatement of Lemma \ref{Lem:variationofA} and assertion (d) is a restatement of (\ref{eq:nutimeskappageq0}) in Lemma \ref{Lem:almosteverywherealong1skeleton}.
\end{proof}

\begin{Remark} \label{Rmk:Alambda}
For any $\lambda > 0$ consider the quantity
\[ A^{(\lambda)} (f_0) := \inf \big\{ \area (f') + \lambda \ell( f' |_{V^{(1)}} )  \;\; : \;\; f' \simeq f_0 \;\; \text{relative to $\partial V$} \big\}. \]
Then all assertions of Proposition \ref{Prop:existenceofVminimizer} hold with $A^{(1)}$ replaced by $A^{(\lambda)}$ (in assertion (a) we have to insert the factor $\lambda$ in front of $\ell (f)$).
This follows by rescaling the metric $g$ by a factor of $\lambda$.
\end{Remark}

\section{Area evolution under Ricci flow} \label{sec:areaevolutionunderRF}
\subsection{Overview}
Let in this section $M$ be a closed $3$-manifold with $\pi_2 (M) = 0$.
Consider a finite simplicial complex $V$ whose faces are denoted by $F_1, \ldots, F_n$ and a continuous map $f_0 : V \to M$ such that $f_0 |_{\partial V}$ is a smooth embedding.

Consider a Ricci flow $(g_t)_{t \in [T_1, T_2]}$ on $M$ such that $\scal_t \geq - \frac{3}{2t}$ on $M$ for all $t \in [T_1, T_2]$.
The goal of this section is to study the evolution of the time dependent quantity
\[ A_t (f_0) : = \inf \big\{ \area_t f' \;\; : \;\; f' \simeq f_0 \;\; \text{relative to $\partial V$} \big\} \]
as introduced in section \ref{sec:existenceofvminimizers}.
We now explain our strategy in this section.
Assume first that for some time $t_0 \in [T_1, T_2]$ there is an embedded minimizer $f : V \to M$ in the homotopy class of $f_0$ (relative to $\partial V$), i.e. $\area_{t_0} f = A_{t_0} (f_0)$.
Then by a simple variational argument, we can conclude that at every edge $E \subset V^{(1)} \setminus \partial V$ the unit vector fields $\nu_E^{(1)}, \ldots, \nu_E^{(v_E)}$ along $f |_E$ that are orthogonal to $f |_E$ and outward pointing tangential to the $v_E$ faces which are adjacent to $E$, satisfy the following identity
\begin{equation} \label{eq:nuadduptozero}
 \nu_E^{(1)} + \ldots + \nu_E^{(v_E)} = 0.
\end{equation}
We can then use Hamilton's method (as presented in the proofs of Propositions \ref{Prop:evolsphere} and \ref{Prop:evolminsurfgeneral}) to compute the time derivative of the area of the minimal disk $f |_{F_j}$ for every $j = 1, \ldots, n$
\begin{equation} \label{eq:derivativeonfaceillustration}
 \frac{d}{dt} \Big|_{t = t_0}  \area_t (f |_{F_j} ) \leq \frac3{4t_0} \area_{t_0} (f |_{F_j} ) + \pi  - \int_{\partial F_j} \big\langle \nu_{\partial F_j}, \kappa_{\partial F_j} \big\rangle.
\end{equation}
Here $\nu_{\partial F_j}$ is the unit vector field which is normal to $f |_{\partial F_j}$ and outward pointing tangential to $f |_{F_j}$ and $\kappa_{\partial F_j}$ is the geodesic curvature of $f |_{\partial F_j}$.
Now we add up these inequalities for $j = 1, \ldots, n$.
The sum of the integrals on the right hand side can be rearranged and grouped into integrals over each edge of $\partial V$.
By (\ref{eq:nuadduptozero}) the integrals over each edge $E \subset V^{(1)} \setminus \partial V$ cancel each other out and we are left with the integrals over edges $E \subset \partial V$.
So
\[ \frac{d}{dt} \Big|_{t = t_0}  \area_t f \leq \frac3{4t_0} \area_{t_0} (f |_{F_j} ) + \pi n  + \sum_{E \subset \partial V} \int_E |\kappa_E |. \]
This implies that in the barrier sense
\begin{equation} \label{eq:wantthisbarrierbound}
 \frac{d}{dt^+} \Big|_{t = t_0}  A_t (f_0) \leq \frac3{4t_0} \area_{t_0} (f |_{F_j} ) + \pi n   + \sum_{E \subset \partial V} \int_E |\kappa_{E} |.
\end{equation}

Unfortunately, as mentioned in section \ref{sec:existenceofvminimizers}, an existence theory for such a minimizer $f$ is hard to come by.
We will however be able to establish the bound (\ref{eq:wantthisbarrierbound}) without the knowledge of this existence using the following trick.
For every $\lambda > 0$ consider the quantity
\[ A_t^{(\lambda)} (f_0) := \inf \big\{ \area_t (f') + \lambda \ell_t ( f' |_{V^{(1)}} )  \;\; : \;\; f' \simeq f_0 \;\; \text{relative to $\partial V$} \big\} \]
as introduced in Remark \ref{Rmk:Alambda}.
It is not hard to see that for each $t \in [T_1, T_2]$
\begin{equation} \label{eq:AlambdaandA0}
 A_t^{(\lambda)} (f_0) \geq A_t (f_0) \qquad \text{and} \qquad \lim_{\lambda \to 0} A_t^{(\lambda)} (f_0) = A_t (f_0).
\end{equation}
The existence theory for a minimizer of $A_t^{(\lambda)}(f_0)$ is far easier and has been carried out in section \ref{sec:existenceofvminimizers}.
Assume for the purpose of clarity that for some time $t_0$ there is an embedded, smooth minimizer $f : V \to M$ for the corresponding minimization problem, i.e. $\area_{t_0} f + \lambda \ell_{t_0} (f |_{V^{(1)}}) = A^{(\lambda)}_{t_0} (f_0)$.
Then identity (\ref{eq:nuadduptozero}) becomes (compare with (\ref{eq:simplenuiskappa}))
\[ \nu^{(1)}_E +  \ldots + \nu^{(v_E)}_E = \lambda \kappa_E. \]
So when adding up inequality (\ref{eq:derivativeonfaceillustration}) for all $j =1, \ldots, n$ and grouping the integrals on the right hand side by edge, we find that luckily the extra term that arises due to this modified identity has the right sign:
\begin{multline*}
 \frac{d}{dt} \Big|_{t = t_0}  \area_t f \leq \frac3{4t_0} \area_{t_0} (f |_{F_j} ) + \pi n  
  + \sum_{E \subset \partial V} \int_E |\kappa_E | 
 - \sum_{E \subset V^{(1)} \setminus \partial V} \int_E \big\langle \lambda \kappa_E, \kappa_E \big\rangle \\
\leq  \frac3{4t_0} \area_{t_0} (f |_{F_j} ) + \pi n 
  + \sum_{E \subset \partial V} \int_E |\kappa_E |.
\end{multline*}
Now choose a function $\lambda : [T_1, T_2] \to (0,1)$ such that $\lambda' (t) < - K_t \lambda(t)$ where $K_t$ is a bound on the Ricci curvature at time $t$.
This is always possible.
Then we can check that
\[ \frac{d}{dt} \Big|_{t = t_0} A^{(\lambda(t))}_t (f_0) \leq \frac3{4t_0} \area_{t_0} (f |_{F_j} ) + \pi n - \sum_{E \subset \partial V} \int_E |\kappa_{\partial F_j} |. \]
Since $\lambda(t)$ can be chosen arbitrarily small, we are able to derive (\ref{eq:wantthisbarrierbound}) using (\ref{eq:AlambdaandA0}).

Note that this is a simplified picture of the arguments that will be presented in the next subsection.
The main difficulty that needs to be overcome stems from the fact that $f : V \to M$ is in general only defined on the $1$-skeleton and not smooth there and that $f$ might have self-intersections.

\subsection{Main part}
In the following Lemma we deduce a bound on a curvature integral over a minimal disk with smooth boundary.
The statement and its proof are similar to parts of the proofs of Propositions \ref{Prop:evolsphere} and \ref{Prop:evolminsurfgeneral}.

\begin{Lemma} \label{Lem:integralboundonsmoothmindisk}
Let $f : D^2 \to M$ be a smooth, harmonic, almost conformal map and set $\gamma = f_{\partial D^2}$.
Denote by $\kappa : S^1 = \partial D^2 \to TM$ the geodesic curvature of $\gamma$ and by $\nu : S^1 \to TM$ the unit vector field along $\gamma$ that is orthogonal to $\gamma$ and outward pointing tangential to $f$ away from possible branch points.
Then
\[ \int_{D^2} \sec^M (df) d{\vol}_{f^*(g)} \geq 2\pi + \int_{S^1} \big\langle \nu(s), \kappa(s) \big\rangle \cdot |\gamma'(s)| ds. \]
Here $\sec^M(df)$ denotes the sectional curvature of $M$ in the direction of the image of $df$.
Note that the integrand on the left hand side is well-defined since the volume form vanishes whenever $df$ is not injective.
\end{Lemma}

\begin{proof}
In order to avoid issues arising from possible branch points (especially on the boundary of $\Sigma$), we employ the following trick (compare with the proof of Proposition \ref{Prop:evolminsurfgeneral}):
Denote by $g_{\textnormal{eucl}}$ the euclidean metric on $D^2$ and consider for every $\varepsilon > 0$ the Riemannian manifold $(D_\varepsilon = D^2, \varepsilon g_{\textnormal{eucl}})$.
The identity map $h_\varepsilon : D^2 \to (D^2, \varepsilon g_{\textnormal{eucl}})$ is a harmonic and conformal diffeomorphism and hence the map $f_\varepsilon = (f, h_\varepsilon) : D^2 \to M \times D_\varepsilon$ is a harmonic and conformal \emph{embedding}.
Denote its image by $\Sigma_\varepsilon = f_\varepsilon(D^2) \subset M \times D_\varepsilon$.
Since the sectional curvatures on the target manifold are bounded, we have
\[ \lim_{\varepsilon \to 0} \int_{\Sigma_\varepsilon} \sec^{M \times D_\varepsilon} ( T \Sigma_\varepsilon) d {\vol} =  \int_{\Sigma}  \sec^M ( df) d {\vol}_{f^*(g)}, \]
where $d{\vol}$ on the left hand side denotes the induced volume form and the integrand denotes the function on $\Sigma_\varepsilon$ that assigns to each point the (ambient) sectional curvature of $M \times D_\varepsilon$ in the direction of its tangent space.

Since $\Sigma_\varepsilon$ is a minimal surface, its interior sectional curvatures are not larger than the corresponding ambient ones.
So we obtain together with Gau\ss-Bonnet
\[ \int_{\Sigma_\varepsilon} \sec^{M \times D_\varepsilon} ( T \Sigma_\varepsilon) d {\vol} \geq \int_{\Sigma_\varepsilon} \sec^{\Sigma_\varepsilon} d{\vol} = 2\pi + \int_{\partial \Sigma_\varepsilon} \kappa^{\Sigma_\varepsilon}_{\partial \Sigma_\varepsilon} d s. \]
Here $\kappa^{\Sigma_\varepsilon}_{\partial \Sigma_\varepsilon}$ denotes the geodesic curvature of the boundary circle viewed as a curve within $\Sigma_\varepsilon$.
We can now estimate the last integral similarly as in the proof of Proposition \ref{Prop:evolminsurfgeneral} (more specifically, see equation (\ref{eq:boundaryintegralsexample})).
Then we obtain that
\[
 \lim_{\varepsilon \to 0} \int_{\partial \Sigma_\varepsilon} \kappa_{\partial \Sigma_\varepsilon}^{\Sigma_\varepsilon} d s =  \int_{S^1} \big\langle \nu(s), \kappa(s) \big\rangle |\gamma'(s) | ds.
\]
This implies the claim.
\end{proof}

Next, we extend the bound of Lemma \ref{Lem:integralboundonsmoothmindisk} to minimal disks that  are bounded by not necessarily embedded, piecewise $C^{1,1}$ loops which satisfy the Douglas-type condition.

\begin{Lemma} \label{Lem:integralcurvboundnotsmooth}
Let $\gamma : S^1 \to M$ be a continuous loop that is a piecewise $C^{1,1}$-immersion and let $\theta_1, \ldots, \theta_p$ be the angles between the right and left derivative of $\gamma$ at the points where $\gamma$ is not differentiable.
(Observe that $\theta_i = 0$ means that both derivatives agree).
Assume that $\gamma$ satisfies the Douglas-type condition (see Definition \ref{Def:DouglasCondition}).
Then there is a solution to the Plateau problem $f : D^2 \to M$ for $\gamma$ which has the following property:

The function $f$ is $C^{1,\alpha}$ up to the boundary away from finitely many points.
Let $\nu : S^1 \to TM$ be the unit vector field along $\gamma$ that is orthogonal to $\gamma$ and outward pointing tangential to $f$ away from possibly finitely many points and let $\kappa : S^1 \to TM$ be almost everywhere equal to the the geodesic curvature of $\gamma$.
Then
\[ \int_{D^2} \sec^M (df) d{\vol}_{f^*(g)} \geq 2\pi  - \theta_1 - \ldots - \theta_p + \int_{S^1} \big\langle \nu(s), \kappa(s) \big\rangle \cdot |\gamma'(s)| ds. \]
\end{Lemma}

\begin{proof}
The proof uses an approximation method.

Let $s_1, \ldots, s_p \in S^1$ be the places where $\gamma$ is not differentiable and choose a small constant $\varepsilon > 0$.
Observe that there is a function $\phi : (0,1) \to (0,1)$ with $\lim_{x \to 0} \phi(x) = 0$ (which may depend on $(M,g)$ and $\gamma$) such that:
We can replace $\gamma$ in a small neighborhood of each $s_i$ by a small arc of length $\leq ( \theta_i + \phi(\varepsilon)) \varepsilon$ and geodesic curvature bounded by $\varepsilon^{-1}$ such that the resulting curve $\gamma^* : S^1 \to M$ is a $C^1$-immersion.
It then follows that if $\kappa^* : S^1 \to M$ is almost everywhere equal to the geodesic curvature of $\gamma^*$, we have
\[ \int_{S^1} \big| \kappa^*(s) - \kappa (s) \big| \cdot |\gamma^{* \prime} (s)| ds \leq \theta_1 + \ldots + \theta_p + p \phi(\varepsilon) + p C \varepsilon. \]
Here $C$ is a $C^{1,1}$ bound on $\gamma$.
Next, we mollify $\gamma^*$ to obtain a smooth immersion $\gamma^{**} : S^1 \to M$ such that if $\kappa^{**} : S^1 \to M$ is the geodesic curvature of $\gamma^{**}$, we have
\[ \int_{S^1} \big| \kappa^{**}(s) - \kappa (s) \big| \cdot |\gamma^{** \prime} (s)| ds \leq \theta_1 + \ldots + \theta_p + p \phi(\varepsilon) + pC \varepsilon + \varepsilon. \]
Finally, we perturb $\gamma^{**}$ to a smooth embedding $\gamma^{***} : S^1 \to M$ whose geodesic curvature $\kappa^{***} : S^1 \to M$ satisfies
\[ \int_{S^1} \big| \kappa^{***}(s) - \kappa (s) \big| \cdot |\gamma^{*** \prime} (s)| ds \leq \theta_1 + \ldots + \theta_p + p \phi(\varepsilon) + pC \varepsilon + 2\varepsilon. \]

These constructions have shown that we can find a sequence $\gamma_1, \gamma_2, \ldots : S^1 \to M$ of smoothly embedded loops with uniform Lipschitz constant that uniformly converge to $\gamma$ and that locally converge on $S^1 \setminus \{ s_1, \ldots, s_q \}$ to $\gamma$ in the $C^{1,\alpha}$ sense such that the geodesic curvatures $\kappa_k : S^1 \to TM$ satisfy
\begin{equation} \label{eq:limsupboundedbythetas}
 \limsup_{k \to \infty} \int_{S^1} \big| \kappa_k (s) - \kappa (s) \big| \cdot |\gamma'_k(s) | ds \leq \theta_1 + \ldots + \theta_q . \end{equation}
Let now $f_1, f_2, \ldots : D^2 \to M$ be solutions of the Plateau problem for these loops.
By Proposition \ref{Prop:PlateauProblem}(b) the maps $f_k$ are smooth up to the boundary.
Moreover, by Proposition \ref{Prop:PlateauProblem}(c) we conclude that, after passing to a subsequence and a possible conformal reparameterization, the maps $f_k : D^2 \to M$ converge uniformly on $D^2$ and smoothly on $\Int D^2$ to a map $f : D^2 \to M$, which solves the Plateau problem for $\gamma$.
By Proposition \ref{Prop:PlateauProblem}(b) the map $f$ has local regularity $C^{1, \alpha}$ up to the boundary away from finitely many points for all $\alpha < 1$.
So by Proposition \ref{Prop:PlateauProblem}(c), the convergence $f_k \to f$ is locally in $C^{1, \alpha}$ away from finitely many points.

We now conclude first that
\begin{equation} \label{eq:interiorcurvboundconvergence}
 \lim_{k \to \infty} \int_{D^2} \sec^M (df_k) d{\vol}_{f^*_k (g)} = \int_{D^2} \sec^M (df) d{\vol}_{f^*(g)}.
\end{equation}
Moreover, if we denote by $\nu_k : S^1 \to M$ the unit normal vectors to $\gamma_k$ that are outward tangential to $f_k$, we obtain that
\begin{equation} \label{eq:limitwithnukbutkappa}
 \lim_{k \to \infty} \int_{S^1} \big\langle \nu_k (s), \kappa(s) \big\rangle \cdot |\gamma'_k(s)| ds= \int_{S^1} \big\langle \nu (s), \kappa(s) \big\rangle \cdot |\gamma'(s)| ds.
\end{equation}
Note also that
\begin{multline*}
 \Big| \int_{S^1} \big\langle \nu_k (s), \kappa_k (s) \big\rangle \cdot |\gamma'_k(s)| ds - \int_{S^1} \big\langle \nu_k (s), \kappa (s) \big\rangle \cdot |\gamma'_k (s)| ds \Big| \\
 \leq \int_{S^1} \big| \kappa_k (s) - \kappa (s) \big| \cdot |\gamma'_k(s) |.
\end{multline*}
Together with (\ref{eq:limsupboundedbythetas}) and (\ref{eq:limitwithnukbutkappa}) this implies 
\[ \liminf_{k \to \infty} \int_{S^1} \big\langle \nu_k (s), \kappa_k (s) \big\rangle \cdot |\gamma'_k(s)| ds \geq  - \theta_1 - \ldots - \theta_q + \int_{S^1} \big\langle \nu (s), \kappa(s) \big\rangle \cdot |\gamma'(s)| ds. \]

Finally, applying Lemma \ref{Lem:integralboundonsmoothmindisk} for each $f_k$, we obtain together with (\ref{eq:interiorcurvboundconvergence}) and the previous estimate that
\begin{multline*}
 \int_{D^2} \sec^M (df) d{\vol}_{f^*(g)} = \lim_{k \to \infty} \int_{D^2} \sec^M (df_k) d{\vol}_{f^*_k (g)} \\
 \geq 2\pi + \liminf_{k \to \infty} \int_{S^1} \big\langle \nu_k (s), \kappa_k (s) \big\rangle \cdot |\gamma'_k(s)| ds \\
 \geq 2\pi - \theta_1 - \ldots - \theta_p + \int_{S^1} \big\langle \nu (s), \kappa(s) \big\rangle \cdot |\gamma'(s)| ds. \qedhere
\end{multline*}
\end{proof}

We can now apply the previous bound together with the results of Proposition \ref{Prop:existenceofVminimizer} to control the time derivative of the quantity $A_t^{(\lambda)}$.
We remark that the proof of this Lemma is again similar to parts of Propositions \ref{Prop:evolsphere}, \ref{Prop:evolminsurfgeneral}.

\begin{Lemma} \label{Lem:evolutionofAlambdaunderRF}
Let $0 < T_1 < T_2 < \infty$ and $(g_t)_{t \in [T_1,T_2)}$ be a smooth solution of the Ricci flow on $M$ on which $\scal_t \geq - \frac{3}{2t}$ for all $t \in [T_1, T_2)$.
Assume that the Ricci curvature of $g_t$ is bounded by some constant $K < \infty$ for all $t \in [T_1, T_2]$.

Let moreover $V$ be a finite simplicial complex whose faces are denoted by $F_1, \ldots, F_n$ and $f_0 : V \to M$ a continuous map such that $f_0 |_{\partial V}$ is a smooth embedding.
At every time $t \in [T_1, T_2)$ let $\gamma_{k, t} : [0, l_{k, t}] \to M$, $(k = 1, \ldots, m_0$) be time-$t$ arclength parameterizations of $f$ restricted to the edges of $\partial V$ and $\kappa_{k, t} : [0, l_{k, t}] \to TM$ the geodesic curvature of each $\gamma_{k, t}$ at time $t$.

Now let $\lambda : [T_1, T_2) \to (0, \infty)$ be a continuously differentiable function such that $\lambda'(t) \leq - K \lambda(t)$ for all $t \in [T_1, T_2)$.
Then we can bound the evolution of the quantity $A^{(\lambda(t))}_t (f_0)$ as follows.
For every $t \in [T_1, T_2)$ we have in the barrier sense:
\[ \frac{d}{dt^+} A^{(\lambda(t))}_t (f_0) \leq \frac{3}{4t} A^{(\lambda(t))}_t (f_0) +  \pi n + \sum_{k=1}^{m_0} \int_0^{l_{k, t}} \big| \kappa_{k, t} (s) \big|_t ds . \]
\end{Lemma}

\begin{proof}
Let $t_0 \in [T_1, T_2]$.
We first apply Proposition \ref{Prop:existenceofVminimizer}(a) (see also Remark \ref{Rmk:Alambda}) at time $t_0$ and obtain a $C^{1,1}$-map $f : V^{(1)} \to M$ that is homotopic to $f_0 |_{V^{(1)}}$ relative $\partial V$ and for which
\[  \sum_{j=1}^n A_{t_0} (f |_{\partial F_j} ) + \lambda(t_0) \ell_{t_0} (f) = A^{(\lambda(t_0))}_{t_0} (f_0). \]
Consider for each $j = 1, \ldots, n$ the loop $f|_{\partial F_j}$ and apply Proposition \ref{Prop:PlateauProblem} to obtain the loops $\gamma_{j,i} : S^1 (l_{j,i}) \to M$.
As in Proposition \ref{Prop:existenceofVminimizer}(b) let $p_{j,i}$ be the number of places where $\gamma_{j,i}$ is not differentiable and let $\kappa_{j,i} : S^1(l_{j,i}) \to TM$ be the geodesic curvature along $\gamma_{j,i}$.
Recall that each $\gamma_{j,i}$ satisfies the Douglas-type condition and that for each $j =1, \ldots, n$
\[ \sum_i A_{t_0} (\gamma_{j,i} ) = A_{t_0} (\gamma_j) \qquad \text{and} \qquad \sum_i (p_{j,i} - 2) \leq 1. \]

Next, we apply Lemma \ref{Lem:integralcurvboundnotsmooth} at time $t_0$ to obtain a solution to the Plateau problem $f_{j,i} : D^2 \to M$ for each $\gamma_{j,i}$ such that for the unit normal vector field $\nu_{j,i} : S^1(l_{j,i}) \to TM$ that is outward pointing tangential to $f_{j,i}$ we have
\[ \int_{D^2} \sec_{t_0}^M (df_{j,i}) d{\vol}_{f_{j,i}^*(g_{t_0})} \geq \pi (2 - p_{j,i}) + \int_{S^1(l_{j,i})} \big\langle \nu_{j,i} (s), \kappa_{j,i} (s) \big\rangle_{t_0}  ds. \]

We can now apply Proposition \ref{Prop:existenceofVminimizer}(c) (or Lemma \ref{Lem:variationofA}) and Proposition \ref{Prop:existenceofVminimizer}(b) to conclude that in the barrier sense for all for all $j = 1, \ldots, n$
\begin{multline*}
 \frac{d}{dt^+} \Big|_{t = t_0} A_t ( f |_{\partial F_j} ) \leq \sum_i \int_{D^2} \frac{d}{dt} \Big|_{t = t_0} d{\vol}_{f^*_{j,i} (g_t)} \\
 = - \sum_i \int_{D^2} \tr_{f^*_{j,i} (g_{t_0})} (\Ric_{t_0} (df_{j,i}, df_{j,i})) d{\vol}_{f^*_{j,i} (g_{t_0})} \displaybreak[1] \\
  =  - \sum_i \bigg( \frac12 \int_{D^2} \big( \scal_{t_0} \circ f_{j,i} \big) d{\vol}_{f^*_{j,i} (g_{t_0}) } + \int_{D_2} \sec^M_{t_0} (df_{j,i}) d{\vol}_{f^*_{j,i} (g_{t_0}) } \bigg) \\
 \leq \frac{3}{4t_0} \sum_i A_{t_0} (\gamma_{j,i}) + \sum_i \pi (p_{j,i} - 2) - \sum_i \int_{S^1(l_{j,i})} \big\langle \nu_{j,i} (s), \kappa_{j,i}(s) \big\rangle ds \\
 \leq \frac{3}{4t_0} A_{t_0} (\gamma_j) + \pi  - \sum_i \int_{S^1(l_{j,i})} \big\langle \nu_{j,i} (s), \kappa_{j,i}(s) \big\rangle ds.
\end{multline*}
Now Proposition \ref{Prop:existenceofVminimizer}(d) implies that if we sum this inequality over all $j = 1, \ldots, n$, then the integral term can be estimated by a boundary integral:
\[ \frac{d}{dt^+} \Big|_{t = t_0} \sum_{j = 1}^n A_t ( f|_{\partial F_j} ) \leq \frac{3}{4t_0} \sum_{j=1}^n A_{t_0} (\gamma_j) + \pi n + \sum_{k = 1}^{m_0} \int_0^{l_{k, t_0}} \big| \kappa_{k, t_0} (s) \big|_{t_0} ds. \]

It remains to estimate the distortion of the length of $f$.
Since the Ricci curvature is bounded by $K$ on $[T_1, T_2]$, we find
\[ \frac{d}{dt}  \Big|_{t = t_0} \big( \lambda(t) \ell_t (f) \big) \leq - K \lambda(t_0) \ell_{t_0} (f) + \lambda(t_0) \cdot K \ell_{t_0} (f) \leq 0. \]
Finally, observe that for all $t \geq t_0$ we have by Lemma \ref{Lem:existenceon1skeleton}
\[ A^{(\lambda(t))}_t (f_0) \leq \sum_{j = 1}^n A_t ( f|_{\partial F_j} ) + \lambda(t) \ell_t (f). \]
The equality is strict for $t = t_0$ and the time derivative of the right hand side is bounded by exactly the desired term in the barrier sense.
This finishes the proof of the Lemma.
\end{proof}

Letting the parameter $\lambda$ go to zero yields the following estimate, which does not require a global curvature bound.

\begin{Lemma} \label{Lem:evolofAunderRF}
Let $0 < T_1 < T_2 \leq \infty$ and $(g_t)_{t \in [T_1,T_2)}$ be a smooth solution of the Ricci flow on $M$ on which $\scal_t \geq - \frac{3}{2t}$ for all $t \in [T_1, T_2)$.

Let moreover $V$ be a finite simplicial complex whose faces are denoted by $F_1, \ldots, F_n$ and $f_0 : V \to M$ a continuous map such that $f_0 |_{\partial V}$ is a smooth immersion.
At every time $t \in [T_1, T_2)$ let $\gamma_{k, t} : [0, l_{k, t}] \to M$, $(k = 1, \ldots, m_0$) be time-$t$ arclength parameterizations of $f_0$ restricted to the edges of $\partial V$ and $\kappa_k : [0, l_{k, t}] \to TM$ the geodesic curvature of each $\gamma_{k, t}$ at time $t$.

Then we can bound the evolution of $A_t (f_0)$ as follows in the barrier sense:
\[ \frac{d}{dt^+} A_t (f_0) \leq \frac{3}{4t} A_t (f_0) +  \pi n + \sum_{k=1}^{m_0} \int_0^{l_{k, t}} \big| \kappa_{k, t} (s) \big|_t ds . \]
\end{Lemma}

\begin{proof}
Note that by a perturbation argument we only need to consider the case in which $f_0 |_{\partial V}$ is an embedding.
Moreover, we can without loss of generality restrict to a time-interval on which the Ricci curvature is bounded by some constant $K < \infty$.
For brevity set
\[ R_t = \pi n + \sum_{k=1}^{m_0} \int_0^{l_{k, t}} \big| \kappa_{k, t} (s) \big|_t ds. \]
Note that $R_t$ is continuous with respect to $t$.
Let $\varepsilon > 0$ be a small constant and apply Lemma \ref{Lem:evolutionofAlambdaunderRF} with $\lambda(t) = \varepsilon \exp(-Kt)$.
We obtain
\[ \frac{d}{dt^+} A^{(\varepsilon \exp(-Kt))}_t (f_0) \leq \frac{3}{4t} A^{(\varepsilon \exp(-Kt))}_t (f_0) + R_t. \]

Let now $t_0 \in [T_1, T_2)$ and consider the solution of the differential equation
\[ \frac{d}{dt} F_{t_0, \varepsilon} (t) = \frac{3}{4t} F_{t_0, \varepsilon} (t) + R_t \qquad \text{and} \qquad F_{t_0, \varepsilon} (t_0) = A_{t_0}^{(\varepsilon \exp (-Kt_0))} (f_0). \]
It follows that
\[  A_t^{(\varepsilon \exp(-K t))} (f_0) \leq F_{t_0, \varepsilon} (t) \qquad \text{for all} \qquad t \geq t_0. \]
Letting $\varepsilon \to 0$ and using the fact that that $\lim_{\lambda \to 0} A^{(\lambda)}_t (f_0) = A_t(f_0)$ yields
\[ A_t (f_0) \leq F_{t_0, 0} (t) \qquad \text{for all} \qquad t \geq t_0 \]
where $F_{t_0, 0}$ satisfies the differential equation
\[ \frac{d}{dt} F_{t_0, 0} (t) = \frac{3}{4t} F_{t_0, 0} (t) + R_t \qquad \text{and} \qquad F_{t_0, 0} (t_0) = A_{t_0} (f_0). \]
So $F_{t_0, 0} (t)$ is a barrier for $A_t (f_0)$ with the required properties.
\end{proof}

We can finally state our third main result.

\begin{Proposition} \label{Prop:areaevolutioninMM}
Let $\MM$ be a Ricci flow with surgery with precise cutoff defined on a time-interval $[T_1, T_2)$ (where $0 < T_1 < T_2 \leq \infty)$, assume that all surgeries are trivial and assume that $\pi_2 (\MM(t)) = 0$ for all $t \in [T_1, T_2)$.
Consider a finite simplicial complex $V$ whose faces are denoted by $F_1, \ldots, F_n$.

Let $f_0 : V \to \MM (T_1)$ be a continuous map such that $f_{0,0} = f_0 |_{\partial V}$ is a smooth immersion.
Consider a smooth family of immersions $f_{0, t} : \partial V \to \MM(t)$ parameterized by time that extend $f_{0,0}$ and that don't meet any surgery points.
Assume moreover that there is a constant $\Gamma < \infty$ such that for each $t \in [T_1, T_2)$ the following is true:
Let $\gamma_{k, t} : [0, l_{k, t}] \to \MM(t)$, $(k = 1, \ldots, m_0$) be time-$t$ arclength parameterizations of $f_{0,t}$ restricted to the edges of $\partial V$ and $\kappa_k : [0, l_{k, t}] \to T\MM(t)$ the geodesic curvature of each $\gamma_{k, t}$ at time $t$.
Then
\[ \sum_{k=1}^{m_0} \int_0^{l_{k, t}} \big( \big| \kappa_{k, t} (s) \big|_t + \big| \partial_t \gamma_{k,t}^\perp (s) \big|_t \big) ds \leq \Gamma. \]
Here $\partial_t \gamma_{k,t}^\perp (s)$ is the component of $\partial_t \gamma_{k,t} (s)$ that is perpendicular to $\gamma_{k,t}$.

For every time $t \in [T_1, T_2)$ denote by $A(t)$ the infimum over the areas of all piecewise smooth maps $f : V \to \MM(t_0)$ such that $f |_{\partial V} = f_{0,t}$ and such that there is a homotopy between $f_0$ and $f$ in space-time that restricts to $f_{0,t'}$ on $\partial V$.

Then the quantity
\[ t^{1/4} \big( t^{-1} A(t) - 4 \pi n - 4 \Gamma \big) \]
is monotonically non-increasing on $[T_1, T_2)$ and if $T_2 = \infty$, we have
\[ \limsup_{t \to \infty} t^{-1} A(t) \leq  4 \pi n + 4 \Gamma. \]
\end{Proposition}

\begin{proof}
Note that the property of having precise cutoff implies that the metric $g(t)$ has $t^{-1}$-positive curvature, which in turn entails that $\scal_t \geq -\frac{3}{2t}$ (see \cite[Definitions \ref{Def:phipositivecurvature}, \ref{Def:precisecutoff}(1)]{Bamler-LT-Perelman}).
Note also that we can mollify each $f : V \to \MM (t)$ that is $C^1$ on $V^{(1)}$ and $V \setminus V^{(1)}$ and that is $W^{1,2}$ on each face of $V$ to a map that is piecewise smooth.
So $A(t) = A_t (f_0)$.

So the monotonicity of the desired quantity away from surgery times follows directly from Lemma \ref{Lem:evolofAunderRF} together with a variational estimate dealing with the fact that $f_{0,t}$ can move in time (similarly as in the proof of Lemma \ref{Lem:variationofV}).
By \cite[Definition \ref{Def:precisecutoff}]{Bamler-LT-Perelman} the the value of $A(t)$ cannot increase under a surgery, i.e. the function $A(t)$ is lower semi-continuous.
\end{proof}

\end{document}